\documentclass{amsart}
\usepackage{amsmath}
\usepackage{amssymb,amsthm}
\usepackage{bm}
\usepackage[dvipdfmx]{graphicx}
\usepackage{mathrsfs}
\numberwithin{equation}{section}
\newtheorem{theorem}{Theorem}[section]
\newtheorem{remark}[theorem]{Remark}

\newtheorem{lemma}[theorem]{Lemma}

\newtheorem{proposition}[theorem]{Proposition}

\subjclass[2010]{35J10; 35J20; 47A75; 65N30}
\keywords{Schr\"odinger equation, eigenvalue problem, secular equation, Rayleigh-Ritz method, augmented plane wave method}
\title[Lower bound for secular equations on divided regions]{Lower bound for solutions to the secular equation in a direct sum of the Sobolev spaces of divided regions}
\author{Sohei Ashida}

\begin{document}
\maketitle

\begin{abstract}
In this paper a lower bound for solutions to the secular equation of the Schr\"odinger equation with basis functions discontinuous on boundaries of divided regions is given. If the functions do not have the discontinuity, the bound reduces to that for the usual Rayleigh-Ritz method. Difference from the usual Rayleigh-Ritz method is bounded by the degree of discontinuity. The result can be regarded as a theoretical basis of the augmented plane wave method for the band structure calculations in solid state physics. The result would be useful also for other electronic eigenvalue problems in which the behavior of the eigenfunction near nuclei is very different from that in the interstitial region, because it allows us to use different basis functions in different regions in contrast with the usual Rayleigh-Ritz method in which we need to use basis functions with inappropriate behaviors in wide regions. The proof is based only on new general results about the Sobolev space, in particular, equivalence of two semi-norms concerned with discontinuity of the functions in a direct sum of the Sobolev spaces. Therefore, the result does not depend on specific forms of basis functions at all. The main ingredient of the proof is an estimate of elements of the orthogonal complement of $H^1_0(\Omega)$ by their boundary values based on a characterization of them as weak solutions to an elliptic equation on a region $\Omega$.
\end{abstract}

\section{Introduction}\label{firstsec}
In this paper we consider the eigenvalue problem (Schr\"odinger equation)
\begin{equation}\label{myeq1.1}
-\Delta u+Vu=Eu,
\end{equation}
for $u\in H^2(\mathbb R^n)$ or $u$ in a certain subspace of $H^2(D)$ on a region $D\subset \mathbb R^n, n\in\mathbb N$, where $\Delta$ is the Laplacian, $V$ is a measurable real-valued function on $\mathbb R^n$ or $D$, and $E$ is an eigenvalue. When $n=1$, exceptionally detailed information of the eigenfunction $u$ and the eigenvalue $E$ can be obtained by considering local solutions of \eqref{myeq1.1} on several intervals and connecting the solutions at the endpoints of the intervals (see e.g. \cite[Section 34]{Sch} and \cite{Ya}). Also for $n\geq 2$ the structures of eigenfunctions and eigenvalues are strongly affected by the local behavior of $V$. In particular, in electronic eigenvalue problems the structures of eigenfunctions of an electron in neighborhoods of nuclei are very different from those in the interstitial region. Therefore, it is desirable to develop methods to construct local solutions in several regions $\Omega_1,\dots,\Omega_N$ and connect the solutions at the boundaries $\Omega_{\alpha}\cap\Omega_{\beta},\ \alpha\neq\beta$ also in the case of $n\geq 2$. However, when $n\geq 2$, connection of local solutions at intersections $\partial\Omega_{\alpha}\cap\partial\Omega_{\beta},\ \alpha\neq \beta$ can not be achieved exactly in general. Actually, exact local solution itself is difficult to obtain in general when $n\geq 2$. Nevertheless, in physical situations we can speculate local basis functions that approximate the local true solutions by their linear combinations. Thus it is expected to obtain approximations of the true eigenfunctions by the linear combinations of the basis functions $w^i_{\alpha},\ i=1,2,\dots$ in different regions $\Omega_{\alpha},\ \alpha=1,\dots,N$. These linear combinations can not be connected exactly on the boundaries $\Omega_{\alpha}\cap\Omega_{\beta},\ \alpha\neq\beta$ as stated above. Instead, in practical calculations a feasible method would be to construct linearly independent functions $u^i\in \bigoplus_{\alpha=1}^NH^2(\Omega_{\alpha}),\ i=1,\dots M$ 
%satisfying $\sum_{\alpha=1}^N\int_{\Omega_{\alpha}} \overline{u^i}(x)u^j(x)dx=\delta_{ij}$
that are linear combinations of the basis functions $w^i_{\alpha}$ on each $\Omega_{\alpha}$ and have only small discontinuity on the boundaries $\Omega_{\alpha}\cap\Omega_{\beta}$ first, and then solve the secular equation given by these functions $u^i$ as in the Rayleigh-Ritz method (see e.g. \cite[Theorem XIII.3]{RS4}) to obtain the energies $\tilde E_m=\tilde E_m(u^1,\dots,u^M),\ m=1,\dots,M$.

Here let us recall the secular equation and the Rayleigh-Ritz method. Let $X$ be a Hilbert space, $A$ be a semibounded selfadjoint operator in $X$, $q$ be the quadratic form associated with $A$ (see the appendix), and $Q\subset X$ be the form domain of $q$. For linearly independent elements $\{u^1,\dots,u^M\}\in Q$ in $X$ we set $h_{ij}:=q(u^i,u^j)$ and $s_{ij}:=\langle u^i,u^j\rangle$. The secular equation for $E\in \mathbb R$ is defined by 
$$\mathrm{det}\, (H-ES)=0,$$
where $H=(h_{ij}),\ S=(s_{ij})$. In particular, if $\{u^i\}$ is orthonormal, $S$ is a unit matrix and $E$ is an eigenvalue of $H$. Let $\tilde E_1,\dots,\tilde E_M\in\mathbb R$ be solutions to the secular equation and $E_1,\dots,E_M\in\mathbb R$ be the $M$ lowest eigenvalues of $A$. Then by the min-max principle (cf. \cite[Theorem XIII.1, 2]{RS4}) we have
\begin{equation}\label{myeq1.1.0.0.1}
\tilde E_m\geq E_m,\ m=1,\dots,M.
\end{equation}
The evaluation of eigenvalues $E_m$ based on this inequality is called the Rayleigh-Ritz method.

For the Schr\"odinger equation the quadratic form is $\langle\nabla u,\nabla v\rangle+\langle u,Vv\rangle$ and the form domain is $H^1(\mathbb R^n)$ or a subspace of $H^1(D)$, where $\langle\cdot,\cdot\rangle$ is the usual inner product in $L^2(\mathbb R^n)$ or $L^2(D)$ and $\langle\nabla u,\nabla v\rangle=\sum_{j=1}^n\langle\partial_{x_j} u,\partial_{x_j} v\rangle$. However, when we use the method in the first paragraph, because of the discontinuity on $\partial\Omega_{\alpha}\cap\partial\Omega_{\beta}$ we have $u^i\notin H^1(\mathbb R^n)$ or $u^i\notin H^1(D)$. Therefore, although the energy $\tilde E_m$ is a solution of the formal secular equation, the min-max principle can not be applied and \eqref{myeq1.1.0.0.1} does not hold. This means that the eigenvalues obtained by the secular equation are not upper bounds of the true eigenvalues, and that $\inf_{\{u^1,\dots,u^M\}}\tilde E_m(u^1,\dots,u^M),\ m=1,\dots,M$ do not coincide with the true eigenvalues in contrast to the Rayleigh-Ritz method. Recall that decrease of the value $\tilde E_m$ obtained from the secular equation in the usual Rayleigh-Ritz method means that the value has been improved, because it is an upper bound.

An existing practical method that can be regarded as an implementation of the idea in the first paragraph is the augmented plane wave (APW) method (see e.g. \cite{GP}) proposed by Slater \cite{Sl} in which the problem as above arises. The APW method is one of the methods of band structure calculations in solid state physics, and its purpose is to obtain eigenvalues of the Bloch functions. The APW method is preferred, because the method converges rapidly and it is independent of the degree of localization of electrons (see e.g. \cite{GP}). There is a huge amount of physical literature about the APW method, but the method is rarely studied from mathematically rigorous theoretical standpoint. 
%Even the selfadjoint realization of the Schr\"odiner operator corresponding to the Bloch function in a unit cell can not be found in the literature. The problem of selfadjointness is that of the regularity up to the boundary of the unit cell of solutions to nonhomogeneous equations in the sense of distribution, and the solution in the sense of distribution is not assumed to have even the first weak derivative. Moreover, the boundary condition for the Bloch function or the periodic boundary condition are different from the Dirichlet or Neumann boundary condition.Therefore, we can not use the method of the global regularity of weak solutions for the Dirichlet or Neumann boundary condition. Due to lack of explicit selfajoint realization and its proof, the setting of the whole problem does not seem to be as clear as in the eigenvalue problems for atoms and molecules.
Even the quadratic form associated with the selfadjoint realization of $-\Delta+V$ in the unit cell can not be found in the literature. As for the mathematical and theoretical convergence analysis of the algorithm, the author is aware of only one result by Chen and Schneider \cite{CS} that proves existence of a convergent subsequence of approximate solutions as the basis set increases and an error estimate for the subsequence (In \cite{CS} the Bloch function of the case $k=0$ in the notation below is considered). Let us consider the Bloch functions in the  periodic lattice with the primitive translation vectors $a_1,\dots,a_n$. The Bloch function is a generalized eigenfunction of $-\Delta+V$ (i.e. a solution to $(-\Delta+V)u=Eu$ with some $E\in\mathbb R$ in the sense of distribution) satisfying the condition that there exists some $k\in \mathbb R^n$ such that $u(x+T)=e^{ik\cdot T}u(x)$ for any $T\in \mathbb R^n$ written as $T=l_1a_1+\dotsm+l_na_n,\ l_1,\dots,l_n\in \mathbb Z$. Under a mild assumption on the potential, we can see that the Bloch function satisfies $u\in H^2_{\mathrm{loc}}(\mathbb R^n)$ and if it is restricted to the unit cell 
\begin{equation}\label{myeq1.1.0.1}
D:=\left\{x=\sum_{i=1}^nc_ia_i:\forall i,\ 0< c_i<1\right\},
\end{equation}
it becomes an eigenfunction of the selfadjoint operator $-\Delta+V$ in $L^2(D)$ with the domain
\begin{equation}\label{myeq1.1.1}
\begin{split}
\mathcal D_k:=\Bigg\{&u\in H^2(D):\forall j,\ u(x+a_j)=e^{ik\cdot a_j}u(x),\\
&(G_j\cdot\nabla) u(x+a_j)=e^{ik\cdot a_j}(G_j\cdot\nabla) u(x),\ \mathrm{for}\ x=\sum_{l\neq j}c_la_l,\ 0< c_l< 1\Bigg\},
\end{split}
\end{equation}
where $G_i\in\mathbb R^n$ is the reciprocal lattice vector such that $G_i\cdot a_j=2\pi\delta_{ij}$ (see the appendix). Here the boundary values are defined in the sense of the trace operator introduced later. The role of the Bloch function as in solid state physics does not seem to have been justified in mathematically rigorous way yet, although it is out of scope of this paper. Expansion of $L^2(\mathbb R^n)$ functions by the Bloch functions was proved by \cite{OK} for continuous potentials, although the proof is incomplete because the dependence of the Bloch functions on $k$ is not considered (cf. \cite[pp.151--152]{Wi} and \cite[pp.615--616]{BLP}). A complete proof owing to \cite{Ge} of the expansion is given in \cite[Theorem XIII.98]{RS4} (see also \cite[Chap.4, Theorem 3.1]{BLP}) under a certain mild assumption on the potential. For the holomorphic dependence of the Bloch functions on $k$ for $L^2$ potentials in the three-dimensional unit cell see \cite{Wi}.

In the APW method we divide $D$ into regions as $D=\Omega_1\cup\left(\bigcup_{\alpha=2}^N\overline \Omega_{\alpha}\right)$,
\begin{align*}
\Omega_{\alpha}&:=\{x:|x-\bar x_{\alpha}|<R_{\alpha}\},\ \alpha=2,\dots,N,\\
\Omega_1&:=D\setminus\left(\bigcup_{\alpha=2}^N\overline\Omega_{\alpha}\right),
\end{align*}
where $\bar x_{\alpha},\ \alpha=2,\dots,N$ are atomic sites and $R_{\alpha}>0$ is small enough so that $\overline\Omega_{\alpha}\subset D$ and $\overline\Omega_{\alpha}\cap\overline\Omega_{\beta}=\emptyset,\ 2\leq\alpha,\beta\leq N,\ \alpha\neq\beta$ (see Figure \ref{fig4}). 
\begin{figure}[h]
\centering
\includegraphics[width=0.8\textwidth, trim={3cm 4.5cm 3cm 1cm}]{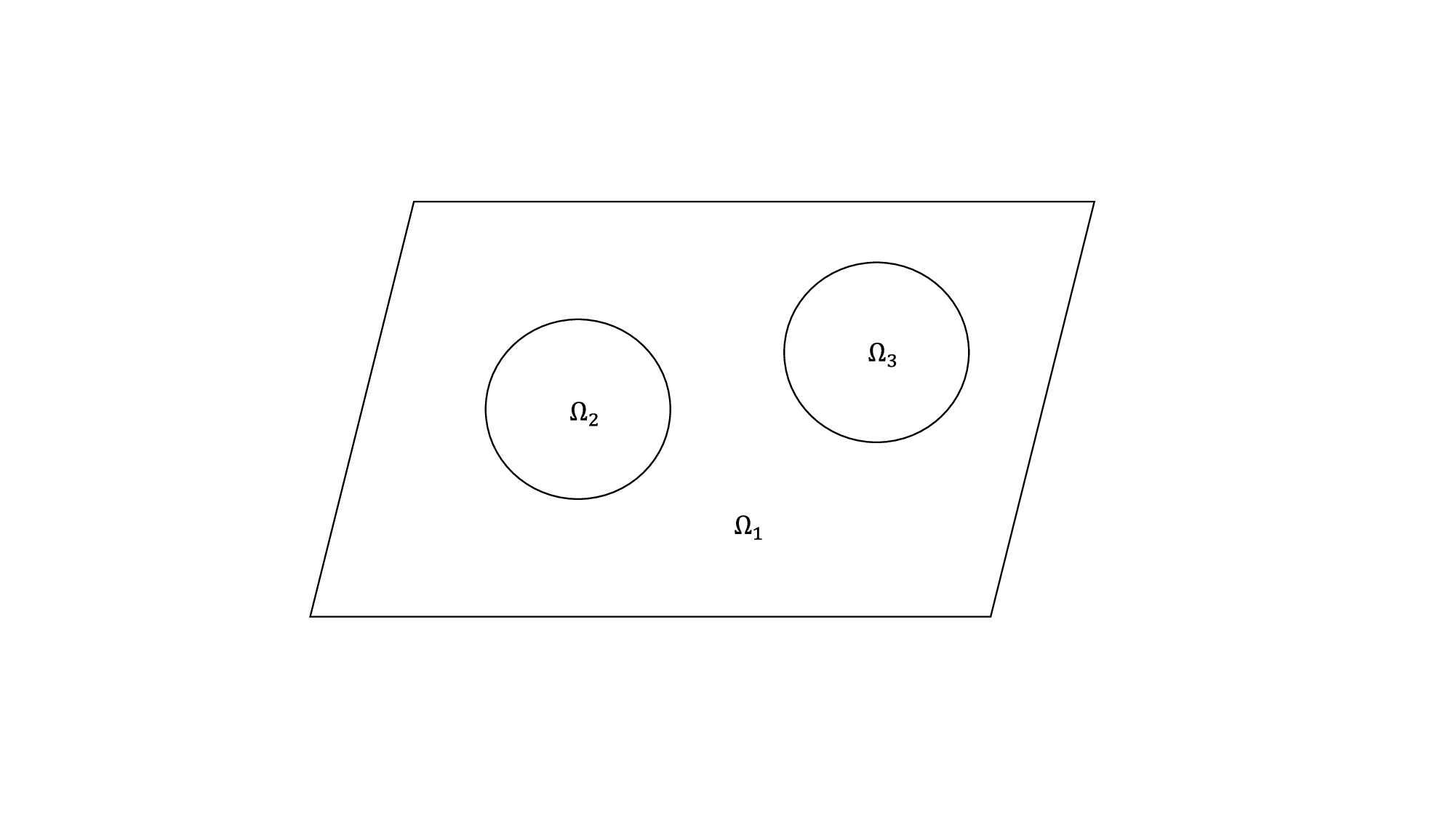}
\caption{Regions in the unit cell}\label{fig4}
\end{figure}
In the usual APW method we assume that the potential has the form
$$V(x)=\begin{cases}
V(|x-\bar x_{\alpha}|),\ &x\in\Omega_{\alpha},\ \alpha\neq 1 \\
0 & x\in\Omega_1
\end{cases},$$
that is, $V$ is spherically symmetric in $\Omega_{\alpha},\ 2\leq \alpha\leq N$. Such a potential is called Muffin-Tin potential. In $\Omega_1$ the eigenfunction $u$ is given by the linear combination of the plane waves $v_G(x)=e^{i(k+G)\cdot x}$, where $G$ is a reciprocal lattice vector i.e. $\forall i,\ G\cdot a_i/(2\pi)\in\mathbb Z$. Note that these functions $v_G(x)$ are local solutions, but they have discrete energies $|k+G|^2$ and we have not constructed local solutions for arbitrary energy $E$. On the other hand, in $\Omega_{\alpha},\ \alpha\geq 2$, we consider the linear combination
$$v_G(x)=e^{i(k+G)\cdot \bar x_{\alpha}}\sum_{l=0}^{\infty}\sum_{m=-l}^lA^{\alpha,k+G}_{lm}\chi^{\alpha}_l(r_{\alpha},E)Y_{lm}(\theta_{\alpha},\varphi_{\alpha}),$$
where $(r_{\alpha},\theta_{\alpha},\varphi_{\alpha})$ is the polar coordinates of $x-\bar x_{\alpha}$, $\chi^{\alpha}_l(r_{\alpha},E)$ is a radial function depending on the energy $E$, $Y_{lm}$ is the spherical harmonics and $A^{\alpha,k+G}_{lm}$ is a constant. $A^{\alpha,k+G}_{lm}$ are determined so that the right hand side conincides with $e^{i(k+G)\cdot x}$ on the boundaries $\partial\Omega_{\alpha}$. In order to determine $A^{\alpha,k+G}_{lm}$ we use the well-known expansion of the plane wave $e^{i(k+G)\cdot x}$ with respect to the center $\bar x_{\alpha}$:
$$e^{i(k+G)\cdot x}=4\pi e^{i(k+G)\cdot \bar x_{\alpha}}\sum_{lm}i^lj_l(r'r_{\alpha})Y^*_{lm}(\theta',\varphi')Y_{lm}(\theta_{\alpha},\varphi_{\alpha}),$$
where $(r_{\alpha},\theta_{\alpha},\varphi_{\alpha})$ and $(r',\theta',\varphi')$ are polar coordinates of $x-\bar x_{\alpha}$ and $k+G$ respectively, and $j_l$ is the spherical Bessel function. From the condition that this expression coincides with the expression of $v_G(x)$ inside $\Omega_{\alpha}$ on the boundary $\partial\Omega_{\alpha}$ it follows that
$$A^{\alpha,k+G}_{lm}=4\pi i^l\frac{j_l(r'R_{\alpha})Y^*_{lm}(\theta',\varphi')}{\chi^{\alpha}_l(R_{\alpha},E)},\quad l=0,1,2,\dots,\quad  m=-l,\dots,l.$$
The eigenfunction $u$ is obtained as a linear combination of $u^i=v_{G_i}$ for different reciprocal lattice vectors $G_i$. Note that $u^i\in \bigoplus_{\alpha=1}^NH^2(\Omega_{\alpha})$ but $u^i\notin H^2(D)$ (cf. Proposition \ref{conectionpro}). Thus they are not eigenfunctions in $H^2(D)$. Actually, the energies of $u^i$ in $\Omega_{\alpha},\ \alpha\geq 2$ and in $\Omega_1$ are different. Note also that the components of $u^i$ in $\Omega_{\alpha},\ \alpha\geq 2$ themselves depend on the energy $E$ through $\chi^{\alpha}_l(r_{\alpha},E)$ and $A^{\alpha,k+G}_{lm}$. We consider the secular equation of such waves $\{u^1,\dots,u^M\}$ and the eigenvalue $E$ is determined from the condition that $E$ is a solution of the equation. Once $E$ is determined, since it is satisfying the secular equation for $u^i$, if $u^i$ were in $H^1(D)$, $E$ would be an upper bound of the true eigenvalue by the min-max principle. However, in practice the infinitely many coefficients $A^{\alpha,k+G}_{lm}$ can not be calculated, and we need to cut off the expansion at some finite number of coefficients. Thus $u^i$ are discontinuous on the boundaries $\partial\Omega_{\alpha}$, and the problem as in the third paragraph arises.
%. Hence although the energy $E$ is a solution of  formal secular equation, the min-max principle that is the theoretical basis of the Rayleigh-Ritz method does not hold for the APW method. This means that the eigenvalues obtained by the secular equation are not upper bounds of the true eigenvalues, and the approximate eigenvalues for $\{u^1,\dots,u^M\}$ that minimize the eigenvalues by the secular equation do not coincide with the true eigenvalues in contrast to the Rayleigh-Ritz method.

In the theory of finite element method there is a method to use functions discontinuous on boundaries of the elements. Such a method is called nonconforming finite element method (cf. \cite[section 10.3]{BS}). There are also other methods such as discontinuous Galerkin (DG) method (cf. \cite{PE}). In some cases error estimates were obtained for nonconforming method and DG method. There is also an error estimate of a DG method in the same situation as that of the APW method (cf. \cite{LC}, actually, rather different methods are called by the same term "discontinuous Galerkin method"). In these analysis specific forms of basis functions play crucial roles.

There would be other problems in which basis functions in divided regions are useful. When we seek critical points of the quadratic form
$$q(u,u):=\langle \nabla u,\nabla u\rangle+\langle u,Vu\rangle,$$
by discretization methods such as linear combination of atomic orbitals (LCAO) method, we need to consider linear combinations of basis functions defined on the whole space $\mathbb R^{n}$. However, these functions are chosen so that they exhibit appropriate behaviors as approximate solutions to the Schr\"odinger equation only in a small region such as a neighborhood of one of nuclear positions. (An example is the Slater type orbital centered at the nucleus.) These functions would exhibit inappropriate behaviors in the other regions. When we use the linear combination of these functions, the functions must contribute in these inappropriate regions with the same coefficient of the linear combination as that in the appropriate small region. This would make the approximation poor. On the other hand, if we consider $q(u,u)$ on $\bigoplus_{\alpha=1}^NH^2(\Omega_{\alpha})$, we can choose different basis functions appropriate for different $\Omega_{\alpha}$ which would be able to highly approximate the solution to the Schr\"odinger equation in the region.
Therefore, it would be very useful if there exists a method to use different basis sets in different divided regions and to estimate the error due to the discontinuity of functions on the intersections $\partial \Omega_{\alpha}\cap \partial\Omega_{\beta},\ \alpha\neq \beta$ of the boundaries of the regions.

The purpose of this paper is to estimate the effect of the discontinuity of functions on the boundaries to the solutions of the secular equation not only in the APW method but also in general eigenvalue problems. Here note that the basis functions in the APW method work only for the special setting. They are plane waves and products of spherical harmonics and radial functions centered at atomic sites. They work for the Muffin-Tin potential, but for general potentials they would not be good choices for the basis functions. Therefore, for the present purpose we need to develop a theory that does not depend on specific structures of the basis functions. The result in this paper is not a convergence analysis. Rather, it is a lower bound for the solutions to the secular equation that allows discontinuity on the boundaries of regions, and it gives an apriori estimate of the difference between the infimum of the solutions to the secular equation as the basis functions change and the true eigenvalue. The result can also be used for a posteriori lower bounds for the solutions. The proof of the main result is based on general results about the Sobolev space obtained in this paper, i.e. Proposition \ref{conectionpro}, Lemmas \ref{orthogonallem}, \ref{decompositionlem}, \ref{decompositionlem2} and in particular,Theorem \ref{distthm} which means that the semi-norm defined by the degree of discontinuity and  that defined by the distance from the Sobolev space in the whole region are equivalent. Thus the main result does not depend on specific forms of basis functions at all. The result can also be regarded as an upper bound of the true eigenvalue, and it should be mentioned that accurate lower bounds to the true eigenvalues can not be obtained, whatever method is chosen even for rather simple systems such as small molecules at present (for methods for accurate lower bounds to eigenvalues see e.g. \cite{We,BG,GHT}). 

The contents of this paper is as follows. In Section \ref{firstsec2} the main results are stated. Several general results are prepared in Section \ref{secondsec}. The main results are proved in Section \ref{thirdsec}. In the appendix we deal with selfadjointness, quadratic forms associated with selfadjoint operators and regularity of the Bloch function.

\section{Main results}\label{firstsec2}
Let us introduce some definitions and assumptions. For a region $\Omega\subset\mathbb R^{n}$ and a natural number $l\geq 1$ we define the Sobolev space $H^l(\Omega)$ by $H^l(\Omega):=\{u\in L^2(\Omega): D^{\gamma}u\in L^2(\Omega)\ \mathrm{for}\ |\gamma|\leq l\}$ which is a Hilbert space with respect to the inner product $(u,v)_l:=\sum_{|\gamma|\leq l}\langle D^{\gamma}u,D^{\gamma}v\rangle_{\Omega}$, where for a multi-index $\gamma=(\gamma_1,\dots,\gamma_{n})\in\mathbb N^{n}$ we define $D^{\gamma}u:=\partial_{x_1}^{\gamma_1}\dotsm\partial_{x_{n}}^{\gamma_{n}}u$ and $\langle \cdot,\cdot\rangle_{\Omega}$ is the usual inner product in $L^2(\Omega)$. We also define for a $C^l$ manifold $\mathcal M$ and a real number $0\leq s\leq l$ the Sobolev space $H^s(\mathcal M)$ using a partition of unity  (cf. \cite{Wl}). For a natural number $l\geq 1$ and a region $\Omega\subset \mathbb R^{n}$ with a bounded boundary (i.e. the boundary is a bounded set as a subset of $\mathbb R^{n}$) of class $C^l$ (see e.g. \cite[Section 6.2]{GT}) there exists a continuous linear operator $T_{l-1}:H^l(\Omega)\to \bigoplus_{j=0}^{l-1}H^{l-j-1/2}(\partial\Omega)$ called a trace operator with the property
$$T_{l-1}u=\left(u|_{\partial\Omega},\dots,\frac{\partial^{l-1}u}{\partial \nu^{l-1}}\bigg|_{\partial\Omega}\right),$$
for $u\in C^l(\overline\Omega)$ (cf. \cite{Wl}), where $u|_{\partial\Omega}$ is the usual boundary value and $\frac{\partial^{j}u}{\partial \nu^{j}}$ is the $j$th derivative in the direction of the inward pointing normal. Note that if $u\in H^l(\Omega),\ l\geq 1$, we have $T_0u\in H^{l-1/2}(\partial\Omega)$ because it is the first component of $T_{l-1}u$. Even for $u\in H^l(\Omega)$ let us denote $T_0u\in H^{l-1/2}(\partial\Omega)$ also by $u|_{\partial \Omega}$, and we sometimes use the notation $u(x)$ for the value of $u|_{\partial\Omega}$ at $x\in \partial\Omega$ (cf. definition of $\mathcal D_k$ in \eqref{myeq1.1.1}).

First we consider the eigenvalue problem in $\mathbb R^n$. We suppose that regions $\Omega_{\alpha}\subset \mathbb R^{n},\ \alpha=1,\dots,N$ satisfy the assumption 
\begin{itemize}
\item[(A)] $\Omega_{\alpha}\cap\Omega_{\beta}=\emptyset, \alpha\neq \beta,\ \bigcup_{\alpha=1}^N\overline{\Omega}_{\alpha}=\mathbb R^{n},\ \partial\Omega_{\alpha}$ is of class $C^3$ and bounded, and each $\partial\Omega_{\alpha}\cap\partial\Omega_{\beta}, \alpha\neq \beta$ is a connected component of $\partial\Omega_{\alpha}$ or the empty set (cf. Figure \ref{omegaarr}).
\end{itemize}
\begin{figure}[h]
\centering
\includegraphics[width=0.8\textwidth, trim={0 2.5cm 2cm 0}]{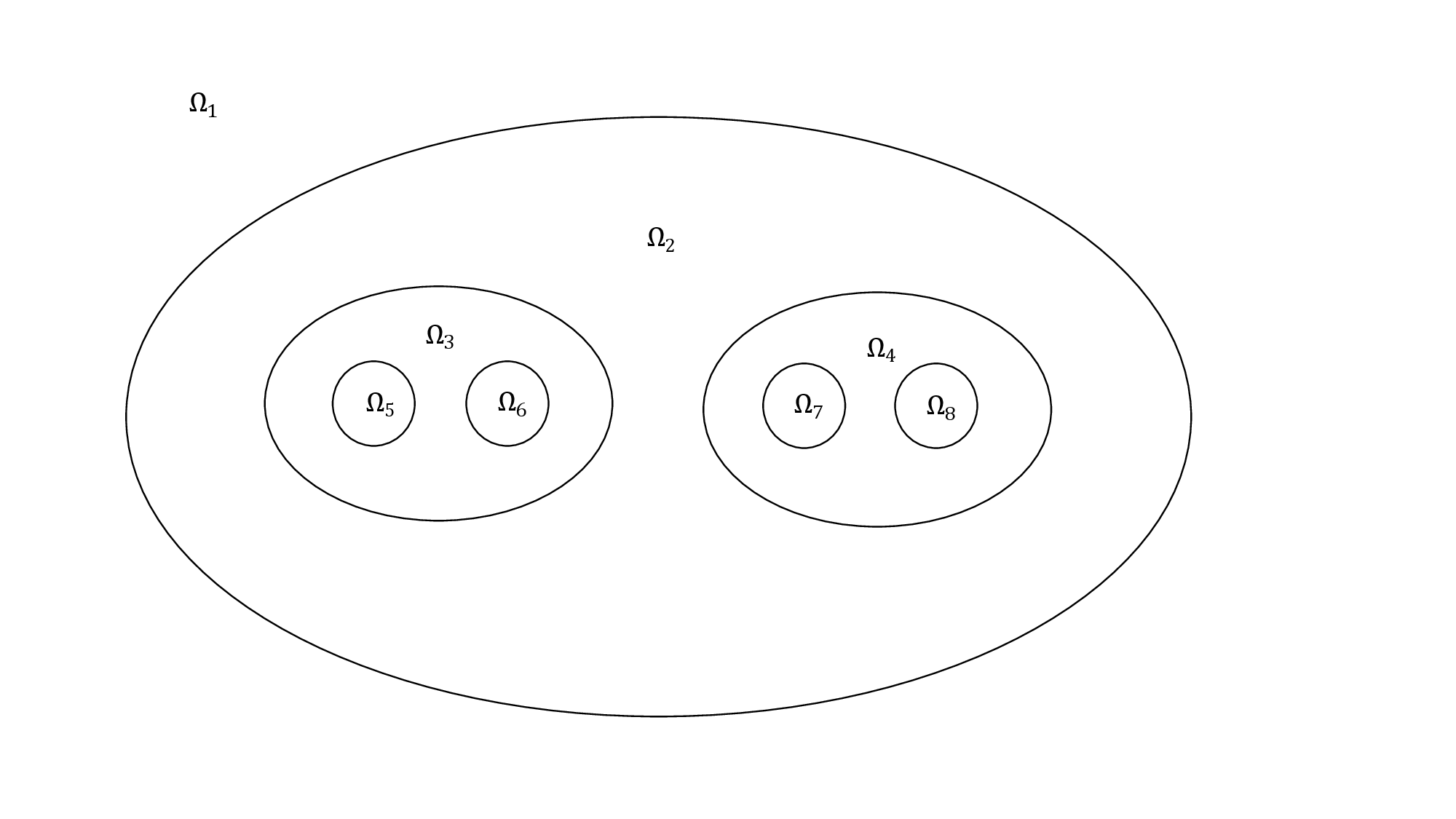}
\caption{Example of the arrangement of regions}\label{omegaarr}
\end{figure}
%For $u\in H^k(\mathbb R^{\nu})$ we define $\tilde u_j\in H^k(\Omega_j)$ by $\tilde u_j(x)=u(x),\ x\in \Omega_j$. We identify $u$ with the tuple $\tilde{\bm u}:=(\tilde u_1,\dots,\tilde u_m)\in \bigoplus_{j=1}^m H^k(\Omega_j)$.
From $u\in H^2(\mathbb R^n)$ we can construct $(\tilde u_1,\dots,\tilde u_N)\in \bigoplus_{\alpha=1}^NH^2(\Omega_{\alpha})$ by
$\tilde u_{\alpha}(x)=u(x),\ x\in\Omega_{\alpha}$. Conversely, for $(\tilde u_1,\dots,\tilde u_N)\in \bigoplus_{\alpha=1}^NH^2(\Omega_{\alpha})$ we define $u\in L^2(\mathbb R^n)$ by
$u(x)=\tilde u_{\alpha}(x),\ x\in \Omega_{\alpha}$. (Note that the Lebesgue measure of $\partial\Omega_{\alpha}$ is $0$, so that the values of $u$ on $\partial\Omega_{\alpha}$ do not matter.) In this sense we have
$$H^2(\mathbb R^n)\subsetneq\bigoplus_{\alpha=1}^NH^2(\Omega_{\alpha})\subsetneq L^2(\mathbb R^n).$$
For the potential we assume
\begin{itemize}
\item[(B)] 
There exist $0\leq a<1$ and $b\geq 0$ such that for any $\tilde u_{\alpha}\in H^1(\Omega_{\alpha})$
\begin{equation*}\label{myeqf2.1}
\langle \tilde u_{\alpha},|V|\tilde u_{\alpha}\rangle_{\Omega_{\alpha}}\leq a\lVert \nabla\tilde u_{\alpha}\rVert_{L^2(\Omega_{\alpha})}^2+b\lVert \tilde u_{\alpha}\rVert_{L^2(\Omega_{\alpha})}^2,\quad \alpha=1,\dots N,
\end{equation*}
and for any $u\in H^2(\mathbb R^n)$
\begin{equation*}\label{myeqf2.2}
\lVert Vu\rVert_{L^2(\mathbb R^n)}\leq a\lVert\Delta u\rVert_{L^2(\mathbb R^n)}+b\lVert u\rVert_{L^2(\mathbb R^n)}.
\end{equation*}
\end{itemize}
The assumption (B) is satisfied e.g. if $V$ is a sum of Coulomb potentials centered at nuclear positions. It follows from (B) that $-\Delta+V$ in $L^2(\mathbb R^n)$ is selfadjoint with the domain $H^2(\mathbb R^n)$. We assume that the eigenvalues of $-\Delta+V$ are bounded from above by a constant $d$ as follows.
\begin{itemize}
\item[(C)] There exist $M\in\mathbb N$ isolated eigenvalues $E_1,\dots, E_M$ of the selfadjoint operator $-\Delta+V$ in $L^2(\mathbb R^n)$ with the domain $H^2(\mathbb R^n)$ in ascending order at the bottom of the spectrum, and there exists a constant $d\geq0$ such that $E_m\leq d,\ m=1,\dots,M$.
\end{itemize}
We denote $\sum_{\alpha=1}^N\langle \tilde u_{\alpha},\tilde v_{\alpha}\rangle_{\Omega_{\alpha}}$ also by $\langle u,v\rangle$ for $u=(\tilde u_1,\dots,\tilde u_N), v=(\tilde v_1,\dots,\tilde v_N)\in \bigoplus_{\alpha=1}^NL^2(\Omega_{\alpha})$. Let $u^i=(\tilde u_1^i,\dots,\tilde u_N^i)\in \bigoplus_{\alpha=1}^NH^2(\Omega_{\alpha}),\ i=1,\dots,M$ be tuples of functions such that $\langle u^i,u^j\rangle=\delta_{ij}$.
We define an $M\times M$ matrix $H=(h_{ij})$ by $h_{ij}:=\langle \nabla u^i,\nabla u^j\rangle+\langle u^i,V u^j\rangle$, where $\nabla u^i:=(\nabla\tilde u^i_1,\dots,\nabla\tilde u^i_N)$, $\langle \nabla u^i,\nabla u^j\rangle:=\sum_{\alpha=1}^N\langle \nabla\tilde u^i_{\alpha},\nabla\tilde u^j_{\alpha}\rangle_{\Omega_{\alpha}}$ and $\langle \nabla\tilde u_{\alpha},\nabla\tilde v_{\alpha}\rangle_{\Omega_{\alpha}}:=\sum_{l=1}^n\langle \partial_{x_l}\tilde u_{\alpha},\partial_{x_l}\tilde v_{\alpha}\rangle_{\Omega_{\alpha}}$. Let $\tilde E_m=\tilde E_m(u^1,\newline\dots,u^M)$ be the $m$th eigenvalue of $H$ in ascending order. The following is a main result.

\begin{theorem}\label{mainthm}
Under the assumptions (A), (B) and (C) there exist constants $C,\delta>0$ depending only on $\{\Omega_{\alpha}\},\ a,\ b$ and $d$ such that if
\begin{equation*}
M\sum_{i=1}^M\sum_{\alpha\neq \beta}\lVert \tilde u^i_{\alpha}|_{\partial{\Omega_{\alpha}}}-\tilde u^i_{\beta}|_{\partial{\Omega_{\beta}}}\rVert_{H^{3/2}(\partial\Omega_{\alpha}\cap\partial\Omega_{\beta})}<\delta,
\end{equation*}
we have
\begin{align*}
&\tilde E_m\geq E_m-CM^{5/2}\sum_{i=1}^M\sum_{\alpha\neq \beta}\lVert \tilde u^i_{\alpha}|_{\partial{\Omega_{\alpha}}}-\tilde u^i_{\beta}|_{\partial{\Omega_{\beta}}}\rVert_{H^{3/2}(\partial\Omega_{\alpha}\cap\partial\Omega_{\beta})},
\end{align*}
for any $1\leq m\leq M$.
\end{theorem}

\begin{remark}\label{mainrem}
(1) The lower bound deteriorates as $M$ increases. However, after solving the secular equation, we obtain the corresponding linear combinations $\varphi^m=\sum_{i=1}^Mc_i^mu^i$ of the basis functions such that $\langle \nabla \varphi^m,\nabla \varphi^m\rangle+\langle \varphi^m,V\varphi^m\rangle=\tilde E_m$, where $(c^m_1,\dots,c_M^m)^T$ is an eigenvector of the matrix $H$. Thus to estimate $\tilde E_{m_0}$ for a fixed $m_0$ a posteriori, we have only to consider $\varphi^1,\dots,\varphi^{m_0}$, and therefore, the lower bound with $M$  and $u^i$ replaced by $m_0$ and $\varphi^i$ holds. In particular, for the a posteriori estimate of $\tilde E_1$, $M$ is always $1$.

(2) Even when $\{u^i\}$ is not orthonormalized, considering the orthonormalization we can see that the same form of estimate holds for the eigenvalues obtained by solving the secular equation. However, for the constant $\delta$ not to be very small and $C$ not to be very large, the discontinuity should not become large by the orthonormalization.

The same remarks hold also for Theorem \ref{mainthm2}.
\end{remark}

Theorem \ref{mainthm} means that the effect of discontinuity of $u^i$ on the boundaries $\partial\Omega_{\alpha}\cap\partial\Omega_{\beta}$ to $\tilde E_m$  is estimated  by the difference of the boundary values $\lVert \tilde u_{\alpha}^i|_{\partial{\Omega_{\alpha}}}-\tilde u_{\beta}^i|_{\partial{\Omega_{\beta}}}\rVert_{H^{3/2}(\partial\Omega_{\alpha}\cap\partial\Omega_{\beta})}$. The infimum of $\tilde E_m$ with respect to $\{u^1,\dots,u^M\}$ is also bounded from above as follows. By $H^2(\mathbb R^n)\subset \bigoplus_{\alpha=1}^NH^2(\Omega_{\alpha})$ choosing eigenfunctions of $-\Delta+V$ in $L^2(\mathbb R^n)$ as $u^1\dots,u^M$ we can see that
$$E_m\geq\inf_{\substack{u^1,\dots,u^M\in \bigoplus_{\alpha=1}^NH^2(\Omega_{\alpha})\\ \langle u^i,u^j\rangle=\delta_{ij}\\ \sum_{i,\alpha\neq \beta}\lVert \tilde u^i_{\alpha}|_{\partial{\Omega_{\alpha}}}-\tilde u^i_{\beta}|_{\partial{\Omega_{\beta}}}\rVert_{H^{3/2}}\leq\epsilon}}\tilde E_m(u^1,\dots,u^M),$$
for any $\epsilon\geq0$. If $\epsilon$ is small enough, combining this and Theorem \ref{mainthm} we obtain
$$E_m\geq\inf_{\substack{u^1,\dots,u^M\in \bigoplus_{\alpha=1}^NH^2(\Omega_{\alpha})\\ \langle u^i,u^j\rangle=\delta_{ij}\\ \sum_{i,\alpha\neq \beta}\lVert \tilde u^i_{\alpha}|_{\partial{\Omega_{\alpha}}}-\tilde u^i_{\beta}|_{\partial{\Omega_{\beta}}}\rVert_{H^{3/2}}\leq\epsilon}}\tilde E_m(u^1,\dots,u^M)\geq E_m-CM^{5/2}\epsilon,$$
which coincides with the principle of the Rayleigh-Ritz method when $\epsilon=0$. This means that even if $u^1,\dots,u^M$ are discontinuous on the boundaries, the solutions to the secular equation can be close to the true eigenvalue by lowering the value $\tilde E_m$ and making the degree of discontinuity $\epsilon$ small.

Next let us consider the case corresponding to the APW method. Let $D$ be the unit cell in \eqref{myeq1.1.0.1} and $k\in\mathbb R^n$ be fixed.
Instead of (A)--(C) we assume (A'), (B') and (C'), as follows.
\begin{itemize}
\item[(A')] $\Omega_{\alpha}\cap\Omega_{\beta}=\emptyset, \alpha\neq \beta,\ \bigcup_{\alpha=1}^N\overline{\Omega}_{\alpha}=\overline{D},\ \partial\Omega_{\alpha}\cap D$ is of class $C^3$, and each $\partial\Omega_{\alpha}\cap\partial\Omega_{\beta}, \alpha\neq \beta$ is a connected component of $\partial\Omega_{\alpha}$ or the empty set. 
\item[(B')] 
There exist $0\leq a<1$ and $b\geq 0$ such that for any $\tilde u_{\alpha}\in H^1(\Omega_{\alpha})$
\begin{equation*}\label{myeqf2.3}
\langle \tilde u_{\alpha},|V|\tilde u_{\alpha}\rangle_{\Omega_{\alpha}}\leq a\lVert \nabla\tilde u_{\alpha}\rVert_{L^2(\Omega_{\alpha})}^2+b\lVert \tilde u_{\alpha}\rVert_{L^2(\Omega_{\alpha})}^2,\quad \alpha=1,\dots N,
\end{equation*}
and for any $u\in \mathcal D_k$ 
\begin{equation*}\label{myeqf2.4}
\lVert Vu\rVert_{L^2(D)}\leq a\lVert\Delta u\rVert_{L^2(D)}+b\lVert u\rVert_{L^2(D)},
\end{equation*}
where $\mathcal D_k$ is defined by \eqref{myeq1.1.1}.
\end{itemize}
The condition (B') is satisfied e.g. if $V$ has only singularities like Coulomb potentials centered at nuclear positions. 
In the appendix we prove that $-\Delta+V$ with the domain $\mathcal D_k$ is a selfadjoint operator under the assumption (B'), so that the following assumption (C') is meaningful.
\begin{itemize}
\item[(C')] There exist $M\in\mathbb N$ isolated eigenvalues $E_1,\dots, E_M$ of the selfadjoint operator $-\Delta+V$ in $L^2(D)$ with the domain $\mathcal D_k$ in ascending order at the bottom of the spectrum, and there exists a constant $d\geq0$ such that $E_m\leq d,\ m=1,\dots,M$.
\end{itemize}
We may assume $\partial\Omega_1\cap\partial D\neq\emptyset$, that is, $\Omega_1$ is the outermost interstitial region. In the appendix we also prove that the operator $-\Delta+V$ on the domain $\mathcal D_k$ is the operator associated with the quadratic form $\langle \nabla u,\nabla u\rangle_D +\langle u,Vu\rangle_D$ with the form domain
\begin{equation*}
Q_k:=\Bigg\{u\in H^1(D): \forall j,\ u(x+a_j)=e^{ik\cdot a_j}u(x),\mathrm{for}\ x=\sum_{l\neq j}c_la_l,\ 0< c_l< 1\Bigg\}.
\end{equation*}
Correspondingly we define
\begin{align*}
\tilde Q_k:=\Bigg\{u=(\tilde u_1,\dots,\tilde u_N)\in \bigoplus_{\alpha=1}^NH^1(\Omega_{\alpha}): &\forall j,\ \tilde u_1(x+a_j)=e^{ik\cdot a_j}\tilde u_1(x),\\
&\mathrm{for}\ x=\sum_{l\neq j}c_la_l,\ 0< c_l< 1\Bigg\}.
\end{align*}
Let $u^i=(\tilde u_1^i,\dots,\tilde u_N^i)\in \tilde Q_k\cap(\bigoplus_{\alpha=1}^NH^2(\Omega_{\alpha})),\ i=1,\dots,M$ be tuples of functions such that $\langle u^i,u^j\rangle=\delta_{ij}$. Define an $M\times M$ matrix $H=(h_{ij})$ and $\tilde E_m,\ m=1,\dots,M$ in the same way as those stated above Theorem \ref{mainthm}. Under these settings we have the result corresponding to Theorem \ref{mainthm}.

\begin{theorem}\label{mainthm2}
Under the assumptions (A'), (B') and (C') there exist constants $C,\delta>0$ depending only on $\{\Omega_{\alpha}\},\ a,\ b$ and $d$ such that if
\begin{equation*}
M\sum_{i=1}^M\sum_{\alpha\neq \beta}\lVert \tilde u^i_{\alpha}|_{\partial{\Omega_{\alpha}}}-\tilde u^i_{\beta}|_{\partial{\Omega_{\beta}}}\rVert_{H^{3/2}(\partial\Omega_{\alpha}\cap\partial\Omega_{\beta})}<\delta,
\end{equation*}
we have
\begin{align*}
&\tilde E_m\geq E_m-CM^{5/2}\sum_{i=1}^M\sum_{\alpha\neq \beta}\lVert \tilde u^i_{\alpha}|_{\partial{\Omega_{\alpha}}}-\tilde u^i_{\beta}|_{\partial{\Omega_{\beta}}}\rVert_{H^{3/2}(\partial\Omega_{\alpha}\cap\partial\Omega_{\beta})},
\end{align*}
for any $1\leq m\leq M$.
\end{theorem}
This result can be regarded as a theoretical basis of the APW method, because it shows that the value $\tilde E_m$ given by the method can be close to the true eigenvalue $E_m$ with the basis functions $u^1,\dots,u^M$ that lower the value obtained by the method, if the discontinuity is small enough. (The discontinuity need to become smaller as $M$ increases, which is accomplished by increasing the basis functions in the regions near the atomic sites.) Recall also that for the a posteriori estimate of $\tilde E_{m_0}$ for some fixed $m_0$, $M$ can be replaced by $m_0$ (cf. Remark \ref{mainrem}. (1)). Here it is important to note the following. Since the basis functions in the APW method themselves depend on the parameter $E$, it is not clear if the increase of the number of basis functions (increase of the number of reciprocal lattice vectors $G$ used for the basis function $v_G$) lowers the value obtained by the method.

The main ingredient of the proof of the results is a characterization of elements in the orthogonal complement $(H_0^1(\Omega))^{\perp}$ of $H_0^1(\Omega)$ in $H^1(\Omega)$ for a bounded region $\Omega$ as weak solutions to an elliptic partial differential equation. Roughly speaking, using a bound of the $H^2(\Omega)$ norm by a norm of the boundary value of the solution (cf. Lemma \ref{orthogonallem}), we can show that the distance between $u\in \bigoplus_{\alpha=1}^NH^2(\Omega_{\alpha})$ (resp., $u\in \tilde Q_k\cap(\bigoplus_{\alpha=1}^NH^2(\Omega_{\alpha}))$) and the space $H^1(\mathbb R^n)\cap(\bigoplus_{\alpha=1}^NH^2(\Omega_{\alpha}))$ (resp., $Q_k\cap(\bigoplus_{\alpha=1}^NH^2(\Omega_{\alpha}))$) is a semi-norm equivalent to $\sum_{\alpha\neq \beta}\lVert \tilde u_{\alpha}|_{\partial{\Omega_{\alpha}}}-\tilde u_{\beta}|_{\partial{\Omega_{\beta}}}\rVert_{H^{3/2}(\partial\Omega_{\alpha}\cap\partial\Omega_{\beta})}$ (cf. Theorem \ref{distthm}). The estimate for the matrix $H$ follows from this equivalence of the semi-norms. In order to apply the estimate for $(H_0^1(\Omega))^{\perp}$ to the proof of the equivalence of the semi-norms, we need to express an element in $\bigoplus_{\alpha=1}^NH^2(\Omega_{\alpha})$ as a sum of functions on regions without holes (cf. Lemma \ref{decompositionlem}). This is achieved by construction of functions using trace operators and right inverses of the trace operators.

%we can estimate the variation of the matrix elements in the Rayleigh-Ritz method in the direction of $(H_0^1(\Omega))^{\perp}\cap H^2(\Omega)$. In order to apply the estimate for $(H_0^1(\Omega))^{\perp}\cap H^2(\Omega)$ to our problem, 

\section{Some preliminaries}\label{secondsec}
The essence of the proof of Theorems \ref{mainthm} and \ref{mainthm2} is the equivalence of two semi-norms concerned with discontinuity of the functions in a direct sum of the Sobolev spaces. For the proof of the equivalence we need to prove several results about the Sobolev space.
Let $\Omega_{\alpha}\subset \mathbb R^{n},\ \alpha=1,\dots,N$ be regions with boundaries of class $C^l$, $l\in \mathbb N,\ l\geq 1$ such that $\Omega_{\alpha}\cap\Omega_{\beta}=\emptyset,\ \alpha\neq \beta$ and $D$ be a region such that $\bigcup_{\alpha=1}^N\overline{\Omega}_{\alpha}=\overline{D}$. We denote by $T^{\alpha}_{l-1}$ the trace operator in $\Omega_{\alpha}$.
First we consider the condition for an element in $\bigoplus_{\alpha=1}^NH^l(\Omega_{\alpha})$ to be a tuple obtained by cutting an element in $H^l(D)$.

\begin{proposition}\label{conectionpro}
Assume that $\tilde u_{\alpha}\in H^l(\Omega_{\alpha}),\ \alpha=1,\dots,N$, and define $u\in L^2(D)$ by $u(x)=\tilde u_{\alpha}(x)$ for $x\in\Omega_{\alpha}$. Then $u\in H^l(D)$ if and only if
$$(T_{l-1}^{\alpha}\tilde u_{\alpha})(x)=(JT_{l-1}^{\beta}\tilde u_{\beta})(x),\ x\in \partial\Omega_{\alpha}\cap\partial\Omega_{\beta},$$
for any $\alpha\neq\beta$ such that $\partial \Omega_{\alpha}\cap\partial\Omega_{\beta}\neq\emptyset$, where $J:\bigoplus_{m=0}^{l-1}H^{l-m-1/2}(\partial\Omega_{\alpha}\cap\partial\Omega_{\beta})\to\bigoplus_{m=0}^{l-1}H^{l-m-1/2}(\partial\Omega_{\alpha}\cap\partial\Omega_{\beta})$ is defined by
$$J\bm g:=(g_0,\dots,(-1)^{m}g_{m},\dots,(-1)^{l-1}g_{l-1}),$$
for $\bm g=(g_0,\dots,g_{l-1})\in \bigoplus_{m=0}^{l-1}H^{l-m-1/2}(\partial\Omega_{\alpha}\cap\partial\Omega_{\beta})$.
\end{proposition}

\begin{remark}
In \cite[Lemma 1.23]{PE} the result corresponding to the case $l=1$ has been given.
\end{remark}

\begin{proof}
Let $\varphi(x)\in C_0^{\infty}(D)$. Using the divergence theorem in $\Omega_{\alpha}$ for the vector valued function $(0,\dots,0,\tilde u_{\alpha}(x)\varphi(x),0,\dots,0)$ whose $i$th component is $\tilde u_{\alpha}(x)\varphi(x)$ we obtain
$$\int_{\Omega_{\alpha}}\tilde u_{\alpha}(x)\frac{\partial\varphi}{\partial x_i}(x)dx=-\int_{\partial\Omega_{\alpha}}\tilde u_{\alpha}|_{\partial\Omega_{\alpha}}(x)\varphi(x)\nu^{\alpha}_i(x)ds-\int_{\Omega_{\alpha}}\frac{\partial\tilde u_{\alpha}}{\partial x_i}(x)\varphi(x)dx,$$
where $\nu_i^{\alpha}(x)$ is the $i$th component of the inward pointing normal vector $\nu^{\alpha}(x)$ of $\partial \Omega_{\alpha}$ at $x\in\partial\Omega_{\alpha}$. Taking sums of both sides with respect to ${\alpha}$ we have
$$\int_{D}u(x)\frac{\partial\varphi}{\partial x_i}(x)dx=-\sum_{{\alpha}=1}^N\int_{\partial\Omega_{\alpha}}\tilde u_{\alpha}|_{\partial\Omega_{\alpha}}(x)\varphi(x)\nu^{\alpha}_i(x)ds-\sum_{{\alpha}=1}^N\int_{\Omega_{\alpha}}\frac{\partial \tilde u_{\alpha}}{\partial x_i}(x)\varphi(x)dx.$$
Thus $u(x)$ is weakly differentiable if and only if $\sum_{{\alpha}=1}^N\int_{\partial\Omega_{\alpha}}\tilde u_{\alpha}|_{\partial\Omega_{\alpha}}(x)\varphi(x)\nu^{\alpha}_i(x)ds=0,\ i=1,\dots,n$ for any $\varphi(x)\in C_0^{\infty}(D)$. For otherwise by considering a sequence $\{\varphi_m\}\subset C_0^{\infty}(D)$ whose support converges to a boundary $\partial\Omega_{\alpha}$, we can see that the right-hand side can not be written as $\int_Dw \varphi dx$ with some $w\in L^1_{\mathrm{loc}}(D)$. Since $\nu_i^{\beta}(x)=-\nu_i^{\alpha}(x)$ at $x\in \partial\Omega_{\alpha}\cap\partial\Omega_{\beta}$, this condition holds if and only if $\tilde u_{\alpha}|_{\partial\Omega_{\alpha}}(x)=\tilde u_{\beta}|_{\partial\Omega_{\beta}}(x),\ x\in\partial\Omega_{\alpha}\cap\partial\Omega_{\beta}$ for any ${\alpha}\neq \beta$ such that $\partial \Omega_{\alpha}\cap\partial\Omega_{\beta}\neq\emptyset$.

As for the higher order derivatives, for any $x\in \partial\Omega_{\alpha}\cap\partial\Omega_{\beta}$ we have a $C^l$-diffeomorphism $\psi$ from a neighborhood $\mathcal N\subset \mathbb R^{n}$ of $x$ to a neighborhood in $\mathbb R^n$ of a point $\psi(x)\in \mathbb R^{n-1}$ such that $\psi(\mathcal N\cap\partial \Omega_{\alpha}\cap\partial\Omega_{\beta})\subset\mathbb R^{n-1}$, where $\mathbb R^{n-1}$ is given by $\mathbb R^{n-1}=\{(x_1,\dots,x_{n-1},0):x_1,\dots,x_{n-1}\in\mathbb R\}$. By an argument using a partition of unity, we have only to consider $\eta\tilde u_{\alpha}$ and $\eta\tilde u_{\beta}$ for some $\eta\in C_0^{\infty}(\mathcal N)$, which we still denote by $\tilde u_{\alpha}$ and $\tilde u_{\beta}$. Moreover, denoting $\tilde u_{\alpha}\circ\psi^{-1}$ and $\tilde u_{\beta}\circ\psi^{-1}$ again by $\tilde u_{\alpha}$ and $\tilde u_{\beta}$, we have only to consider the functions $\tilde u_{\alpha}$ and $\tilde u_{\beta}$ defined on $\Omega_{\alpha}$ and $\Omega_{\beta}$ such that $\partial\Omega_{\alpha}\cap\partial\Omega_{\beta}\subset\mathbb R^{n-1}$. 

Approximating elements in $H^l(\Omega_{\alpha})$ by those in $C^{\infty}(\overline{\Omega}_{\alpha})$ we can see that $\frac{\partial^2\tilde u_{\alpha}}{\partial x_{i_1}\partial x_{i_2}}\newline=\frac{\partial^2\tilde u_{\alpha}}{\partial x_{i_2}\partial x_{i_1}},\ 1\leq i_1,i_2\leq n$ and $\frac{\partial}{\partial x_{i}}(\tilde u_{\alpha}|_{\partial\Omega_{\alpha}\cap\partial\Omega_{\beta}})=\frac{\partial \tilde u_{\alpha}}{\partial x_{i}}\Big|_{\partial\Omega_{\alpha}\cap\partial\Omega_{\beta}},\ 1\leq i\leq n-1$ for $\tilde u_{\alpha}\in H^l(\Omega_{\alpha}),\ l\geq 2$ in the sense of weak derivatives. Hence we can rewrite boundary value of any derivative in the form $\partial_{x_1}^{\gamma_1}\dotsm\partial_{x_{n-1}}^{\gamma_{n-1}}(\partial_{x_{n}}^{\gamma_{n}}\tilde u_{\alpha}|_{\partial\Omega_{\alpha}\cap\partial\Omega_{\beta}})$. By application to $(D^{\gamma}\tilde u_{\alpha})|_{\partial\Omega_{\alpha}\cap\partial\Omega_{\beta}}$ and $(D^{\gamma}\tilde u_{\beta})|_{\partial\Omega_{\alpha}\cap\partial\Omega_{\beta}}$ of the same argument as that for the boundary value $\tilde u_{\alpha}|_{\partial\Omega_{\alpha}}$ above we can see that $u$ has higher order weak derivatives if and only if $\partial_{x_{n}}^{\gamma_{n}}\tilde u_{\alpha}|_{\partial\Omega_{\alpha}\cap\partial\Omega_{\beta}}=\partial_{x_{n}}^{\gamma_{n}}\tilde u_{\beta}|_{\partial\Omega_{\alpha}\cap\partial\Omega_{\beta}}, x\in\partial\Omega_{\alpha}\cap\partial\Omega_{\beta}$. Finally noting that one of the derivations $\frac{\partial}{\partial \nu^{\alpha}}$ and $\frac{\partial}{\partial \nu^{\beta}}$ in the directions of the inward pointing normals is $\partial_{x_{n}}$ and the other is $-\partial_{x_{n}}$, we obtain the result.
\end{proof}

For the proof of Theorems \ref{mainthm} and \ref{mainthm2} the following lemma about a characterization of the intersection of $H^2(\Omega)$ and the orthogonal complement $(H_0^1(\Omega))^{\perp}$ of $H_0^1(\Omega)$ in $H^1(\Omega)$ for a region $\Omega$ with a bounded boundary is crucial.

\begin{lemma}\label{orthogonallem}
Let $\Omega$ be a region with a bounded boundary of class $C^3$. Then there exists a constant $C>0$ depending only on $\Omega$ such that
\begin{equation}\label{myeq2.0}
\lVert u\rVert_{H^2(\Omega)}\leq C\lVert u|_{\partial \Omega}\rVert_{H^{3/2}(\partial \Omega)},
\end{equation}
for any $u\in (H_0^1(\Omega))^{\perp}\cap H^2(\Omega)$.
\end{lemma}

\begin{proof}
By definition, $u\in H^1(\Omega)$ belongs to $(H_0^1(\Omega))^{\perp}$ if and only if
\begin{equation*}\label{myeq2.1}
\langle \nabla v,\nabla u\rangle_{\Omega}+\langle v,u\rangle_{\Omega}=0,
\end{equation*}
for any $v\in H_0^1(\Omega)$. 
This means that $u\in (H_0^1(\Omega))^{\perp}$ if and only if $u$ is a weak solution of $-\Delta u+u=0$. Since $\Omega$ has a boundary of class $C^3$, the continuous linear right inverse $Z_1$ of the trace operator $T_1$ exists, i.e.  $T_1Z_1$ is the identity operator (cf. \cite[Theorem 8.8]{Wl}). Let us define $\bm g:=(T_0u,0)\in H^{3/2}(\partial\Omega)\bigoplus H^{1/2}(\partial\Omega)$. Then we have $Z_1 \bm g\in H^2(\Omega)$ and $T_0(u-Z_1 \bm g)=0$.  Here it is known that $w\in H^1(\Omega)$ satisfies $w\in H^1_0(\Omega)$, if and only if $T_0w=0$ (cf. \cite[Theorem 8.9]{Wl}). Therefore, we have $u-Z_1 \bm g\in H^2(\Omega)\cap H^1_0(\Omega)$. We denote by $-\Delta^D$ the selfadjoint operator in $L^2(\Omega)$ associated with the quadratic form $\langle \nabla u,\nabla u\rangle_{\Omega}$ with the form domain $H^1_0(\Omega)$ (cf. \cite[Definition in p.263]{RS4} and \cite[Theorem VIII.15]{RS}), where the superscript $D$ indicates the Dirichlet boundary condition. When we determine the domain of $-\Delta^D$, using the regularity theorem for elliptic equations (see e.g. \cite[Theorem 8.12]{GT}) we can see that the domain of $-\Delta^D$ is $\mathcal D(-\Delta^D)=H^2(\Omega)\cap H^1_0(\Omega)$, and thus $u-Z_1\bm g\in \mathcal D(-\Delta^D)$. Moreover, since $u, Z_1\bm g\in H^2(\Omega)$ and $u$ is a weak solution of $-\Delta u+u=0$, we have 
\begin{equation}\label{myeq2.1.0.0.1}
(-\Delta^D+1)(u-Z_1\bm g)=(-\Delta+1)(u-Z_1\bm g)=-(-\Delta+1)Z_1\bm g,
\end{equation}
where $-\Delta$ means the usual weak derivative.
Since $-\Delta^D\geq 0$, the real number $-1$ belongs to the resolvent set of $-\Delta^D$. Applying the resolvent $R(-1):=(-\Delta^D+1)^{-1}$ to the both sides of \eqref{myeq2.1.0.0.1}, we obtain $u=Z_1\bm g-R(-1)(-\Delta+1)Z_1\bm g$. Therefore, the estimate \eqref{myeq2.0} follows from the continuity of $Z_1:H^{3/2}(\partial \Omega)\bigoplus H^{1/2}(\partial\Omega)\to H^2(\Omega)$ and $R(-1):L^2(\Omega)\to H^2(\Omega)\cap H^1_0(\Omega)$ (which follows from the open mapping theorem). This completes the proof.
\end{proof}

We have a decomposition $H^2(\Omega_{\alpha})=(H^1_0(\Omega_{\alpha})\cap H^2(\Omega_{\alpha}))\bigoplus ((H^1_0(\Omega_{\alpha}))^{\perp}\cap H^2(\Omega_{\alpha}))$ (cf. proof of Theorem \ref{distthm}). Moreover, we obviously have
$$\sum_{\alpha\neq \beta}\lVert \tilde u_{\alpha}|_{\partial \Omega_{\alpha}}-\tilde u_{\beta}|_{\partial \Omega_{\beta}}\rVert_{H^{3/2}(\partial\Omega_{\alpha}\cap\partial\Omega_{\beta})}\newline=0,$$
for $(\tilde u_1,\dots,\tilde u_N)\in \bigoplus_{\alpha=1}^N(H^1_0(\Omega_{\alpha})\cap H^2(\Omega_{\alpha}))$. Therefore, one would expect that the norm of $(\tilde u_1,\dots,\tilde u_N)\in \bigoplus_{\alpha=1}^N((H^1_0(\Omega_{\alpha}))^{\perp}\cap H^2(\Omega_{\alpha}))$ could be estimated using $\sum_{\alpha\neq \beta}\lVert \tilde u_{\alpha}|_{\partial \Omega_{\alpha}}-\tilde u_{\beta}|_{\partial \Omega_{\beta}}\rVert_{H^{3/2}(\partial\Omega_{\alpha}\cap\partial\Omega_{\beta})}$, which is the point of departure of the idea for Theorem \ref{distthm}. However, this is not the case, because on $\partial\Omega_{\alpha}\cap\partial\Omega_{\beta}$ the values of $\tilde u_{\alpha}|_{\partial \Omega_{\alpha}}$ and $\tilde u_{\beta}|_{\partial \Omega_{\beta}}$ can cancel out. In order to avoid this, we decompose an element in $\bigoplus_{\alpha=1}^NH^2(\Omega_{\alpha})$ into a sum of tuples that have nonzero difference $\tilde u_{\alpha}|_{\partial \Omega_{\alpha}}-\tilde u_{\beta}|_{\partial \Omega_{\beta}}$ only on different  $\partial\Omega_{\alpha}\cap\partial\Omega_{\beta}$ on which the difference is given by $\tilde u_{\alpha}|_{\partial \Omega_{\alpha}}$ because $\tilde u_{\beta}|_{\partial \Omega_{\beta}}=0$, which is the goal of Lemma \ref{decompositionlem} below.

Suppose that $\Omega_{\alpha},\ \alpha=1,\dots,N$ are regions satisfying (A) or (A'). Then for any $1\leq \alpha\leq N$ there exists a region $\hat \Omega^{\alpha}\subset \mathbb R^n$ satisfying the following condition. In the case of (A) (resp., (A')) if $\Omega_{\alpha}$ is bounded (resp., $\partial\Omega_{\alpha}\cap\partial D=\emptyset$), then $\Omega_{\alpha}\subset\hat\Omega^{\alpha}$, $\partial\hat\Omega^{\alpha}\subset\partial\Omega_{\alpha}$ and $\partial\hat\Omega_{\alpha}$ has only one connected component. If $\Omega_{\alpha}$ is unbounded (resp., $\partial\Omega_{\alpha}\cap\partial D\neq\emptyset$), then $\hat\Omega^{\alpha}=\mathbb R^n$ (resp., $\hat\Omega^{\alpha}=D$) in the case of (A) (resp., (A')). Intuitively, $\hat\Omega^{\alpha}$ is the region obtained by filling up the holes in $\Omega_{\alpha}$ (see Figure \ref{fig2}).
\begin{figure}[h]
\centering
\includegraphics[width=0.9\textwidth, trim={0 1cm 2cm 0}]{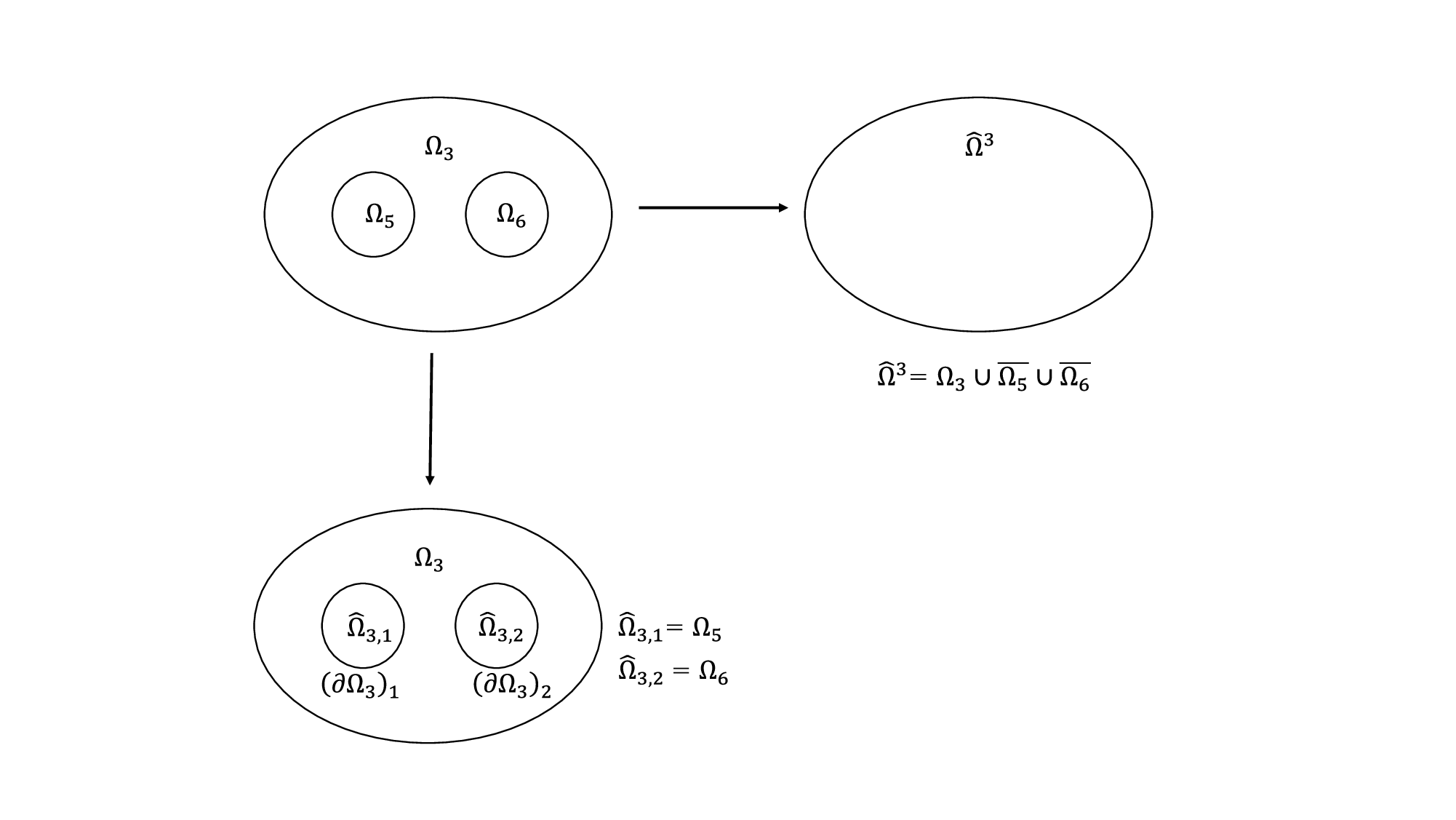}
\caption{Definitions of the regions}\label{fig2}
\end{figure}
\begin{figure}[h]
\centering
\includegraphics[width=0.8\textwidth, trim={0 1cm 2cm 0}]{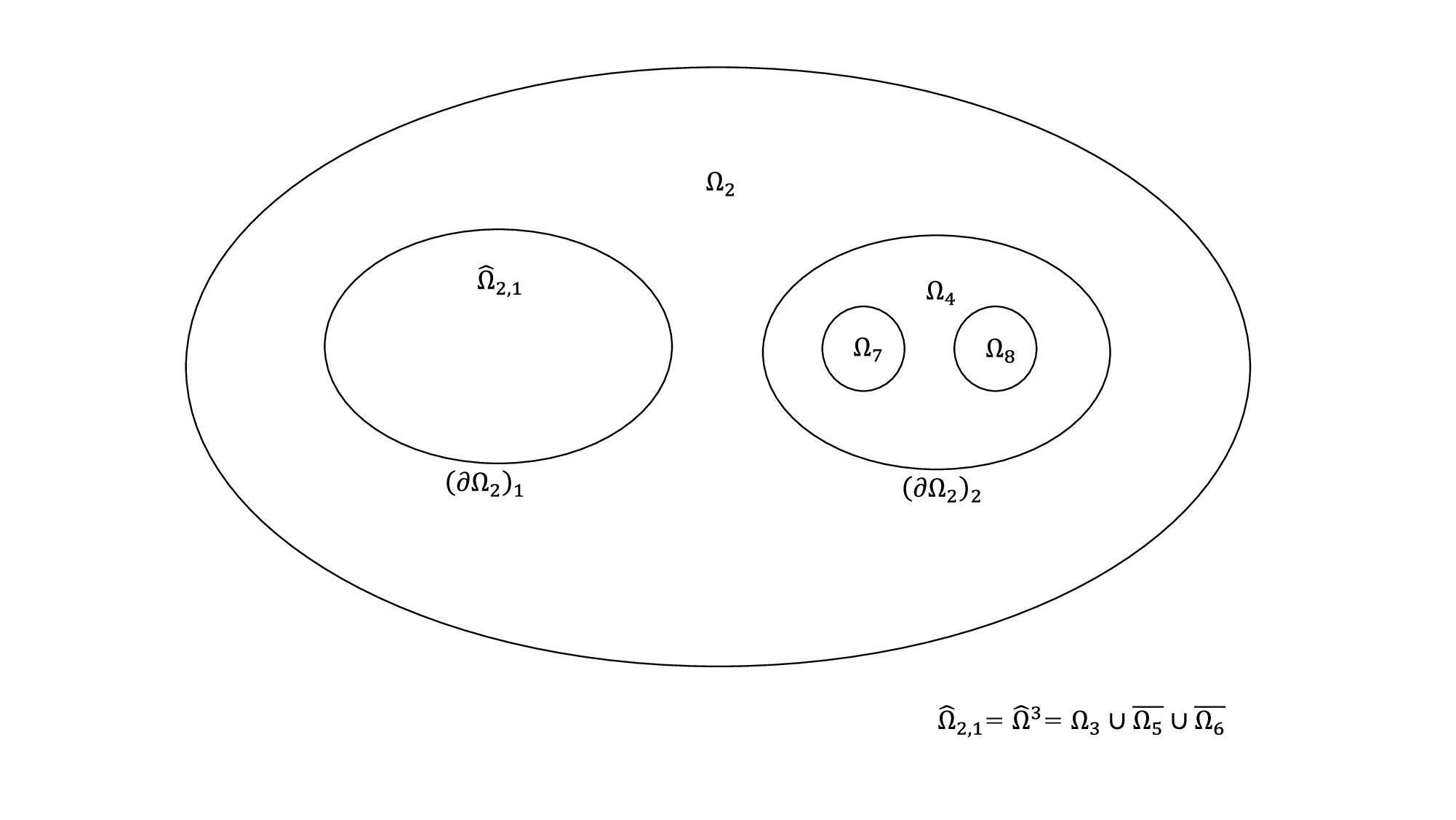}
\caption{Definitions of the regions corresponding to Figure \ref{omegaarr}}\label{fig3}
\end{figure}
For example, $\hat\Omega^3=\Omega_3\cup\overline{\Omega_5}\cup\overline{\Omega_6}$ in the case of Figure \ref{fig2}, and $\hat\Omega^2=\Omega_2\cup\overline{\Omega_3}\cup\overline{\Omega_5}\cup\overline{\Omega_6}\cup\overline{\Omega_4}\cup\overline{\Omega_7}\cup\overline{\Omega_8}$ in the case of Figure \ref{omegaarr}. We denote connected components of $\partial\Omega_{\alpha}$ that enclose bounded regions not including points in $\Omega_{\alpha}$ by $(\partial\Omega_{\alpha})_1,\dots,(\partial\Omega_{\alpha})_{P_{\alpha}}$ (if such connected components exist), and denote the bounded region in $\mathbb R^n$ whose boundary is $(\partial\Omega_{\alpha})_p$ by $\hat\Omega_{\alpha,p}$. Intuitively, $\hat\Omega_{\alpha,p},\ p=1,\dots,P_{\alpha}$ are holes in $\Omega_{\alpha}$ (see Figures \ref{fig2} and \ref{fig3}). Then there exists some $1\leq \beta\leq N$ such that $\hat\Omega^{\beta}=\hat\Omega_{\alpha,p}$. For example, $\hat\Omega_{2,1}=\hat\Omega^3=\Omega_3\cup\overline{\Omega_5}\cup\overline{\Omega_6}$ in the case of Figure \ref{omegaarr} (see Figure \ref{fig3}). We shall consider the case of (A).
Let $u_{\alpha}\in H^2(\hat\Omega^{\alpha})$. If we set $\tilde u_{\beta}\in H^2(\Omega_{\beta})$ by $\tilde u_{\beta}(x)=u_{\alpha}(x),\ x\in \Omega_{\beta}$ for $\Omega_{\beta}\subset\hat \Omega^{\alpha}$ and $\tilde u_{\beta}=0$ for $\Omega_{\beta}\not\subset \hat\Omega^{\alpha}$, we can identify $u_{\alpha}\in H^2(\hat\Omega^{\alpha})$ with $(\tilde u_1,\dots,\tilde u_N)\in \bigoplus_{\beta=1}^NH^2(\Omega_{\beta})$. In this sense we have
$$H^2(\hat\Omega^{\alpha})\subset\bigoplus_{\beta=1}^NH^2(\Omega_{\beta}),$$
which yields
$$\sum_{\alpha=1}^NH^2(\hat\Omega^{\alpha})\subset\bigoplus_{\beta=1}^NH^2(\Omega_{\beta}).$$
In fact, the opposite inclusion holds.

\begin{lemma}\label{decompositionlem}
Under the assumption (A) we have
$$\sum_{\alpha=1}^NH^2(\hat\Omega^{\alpha})=\bigoplus_{\beta=1}^NH^2(\Omega_{\beta}).$$
\end{lemma}

\begin{proof}
We may assume that the regions are labeled so that if $\Omega_{\alpha}$ is included in some region $\hat\Omega_{\beta,p}$, we have $\alpha>\beta$  (see Figure \ref{omegaarr}). Let $u=(\tilde u_1,\dots,\tilde u_N)\in \bigoplus_{\beta=1}^NH^2(\Omega_{\beta})$. We shall construct $v_{\alpha}\in H^2(\hat \Omega^{\alpha}),\ 1\leq \alpha\leq N$ so that $u=\sum_{\alpha=1}^Nv_{\alpha}$. We denote the restriction of the trace operator in $\Omega_{\alpha}$ to a connected component $(\partial\Omega_{\alpha})_p$ by $T^{\alpha,p}_1:H^2(\Omega_{\alpha})\to H^{3/2}((\partial\Omega_{\alpha})_p)\bigoplus H^{1/2}((\partial\Omega_{\alpha})_p)$. We also denote the right inverse of the trace operator in $\hat\Omega_{\alpha,p}$ by $Z_1^{\alpha,p}:H^{3/2}((\partial\Omega_{\alpha})_p)\bigoplus H^{1/2}((\partial\Omega_{\alpha})_p)\to H^2(\hat\Omega_{\alpha,p})$.

We construct auxiliary functions $\tilde u^{\alpha}_{\beta}\in H^2(\Omega_{\beta}),\ 1\leq \beta\leq N$ inductively with respect to $\alpha=0,\dots,N$. We define $\tilde u_{\beta}^0\in H^2(\Omega_{\beta})$ by $\tilde u_{\beta}^0:=u_{\beta}$. Assume that $\tilde u^{\alpha}_{\beta},\ 1\leq \beta\leq N$ have been constructed. We construct $v_{\alpha+1}\in H^2(\hat \Omega^{\alpha+1})$ by
$$v_{\alpha+1}(x):=\begin{cases}
\tilde u^{\alpha}_{\alpha+1}(x),& x\in\Omega_{\alpha+1}\\
(Z^{\alpha+1,p}_1JT^{\alpha+1,p}_1\tilde u_{\alpha+1}^{\alpha})(x),& x\in\hat\Omega_{\alpha+1,p},\ p=1,\dots,P_{\alpha+1}
\end{cases}.$$
where $J$ is the operator defined in Proposition \ref{conectionpro}.
Since the images of $\tilde u^{\alpha}_{\alpha+1}$ and $Z^{\alpha+1,p}_1JT^{\alpha+1,p}_1\tilde u^{\alpha}_{\alpha+1}$ on $(\partial\Omega_{\alpha+1})_p$ by the trace operators are $T^{\alpha+1,p}_1\tilde u_{\alpha+1}^{\alpha}$ and $JT^{\alpha+1,p}_1\newline\tilde u^{\alpha}_{\alpha+1}$, by Proposition \ref{conectionpro} we can see that $v_{\alpha+1}\in H^2(\hat\Omega^{\alpha+1})$. We also set $\tilde u_{\beta}^{\alpha+1}(x)=\tilde u_{\beta}^{\alpha}(x)-v_{\alpha+1}(x),\ x\in\Omega_{\beta}$ (recall that we assume $v_{\alpha+1}(x)=0$ for $x\notin \hat\Omega^{\alpha+1}$).
%$$\tilde u^{\alpha+1}_{\beta}(x):=\begin{cases}
%\tilde u_{\beta}^{\alpha}(x)-(Z^{\alpha+1,p}_1JT^{\alpha+1,p}_1\tilde u_{\alpha+1}^{\alpha})(x)& (\mathrm{if}\ \Omega_{\beta}\subset\hat\Omega_{\alpha+1,p})\\
%0 &(\mathrm{if}\ \beta=\alpha+1)\\
%\tilde u_{\beta}^{\alpha}(x)& (\mathrm{if}\ \beta\neq \alpha+1\\ & \mathrm{and}\ \Omega_{\beta}\not\subset\bigcup_{p=1}^{P_{\alpha+1}}\hat\Omega_{\alpha+1,p})
%\end{cases}.$$

By this procedure we construct $v_1,\dots,v_N$.
Regarding $v_{\alpha}\in H^2(\hat\Omega^{\alpha})$ as elements in $\bigoplus_{\beta=1}^NH^2(\Omega_{\beta})$ as stated above the lemma we have
$$\sum_{\alpha'=1}^{\alpha+1}v_{\alpha'}+(\tilde u^{\alpha+1}_1,\dots,\tilde u^{\alpha+1}_N)=\sum_{\alpha'=1}^{\alpha}v_{\alpha'}+(\tilde u^{\alpha}_1,\dots,\tilde u^{\alpha}_N).$$
Hence it follows that $\sum_{\alpha'=1}^Nv_{\alpha'}=(\tilde u_1,\dots,\tilde u_N)$. This completes the proof.
\end{proof}

The following lemma follows in exactly the same way noting that $v_1(x)=\tilde u_1(x)$, $x\in\Omega_1$ holds for the outermost region $\Omega_1$, if we construct $v_1(x)$ as in the proof.

\begin{lemma}\label{decompositionlem2}
Under the assumption (A') we have
$$\sum_{\alpha=1}^N\tilde Q_k\cap(H^2(\hat\Omega^{\alpha}))=\tilde Q_k\cap(\bigoplus_{\beta=1}^NH^2(\Omega_{\beta})).$$
\end{lemma}

\begin{remark}
The following example clearly illustrates the proof corresponding to the case of Lemma \ref{decompositionlem2} with $k=0$. Set $\Omega_1=\{x:|x|>2,\ |x_i|<3,\ i=1,\dots,n\}$, $\Omega_2=\{x:1<|x|<2\}$ and $\Omega_3=\{x: |x|<1\}$ (see Figure \ref{fig5}).
\begin{figure}[h]
\centering
\includegraphics[width=0.8\textwidth, trim={2cm 2.5cm 2.5cm 1cm}]{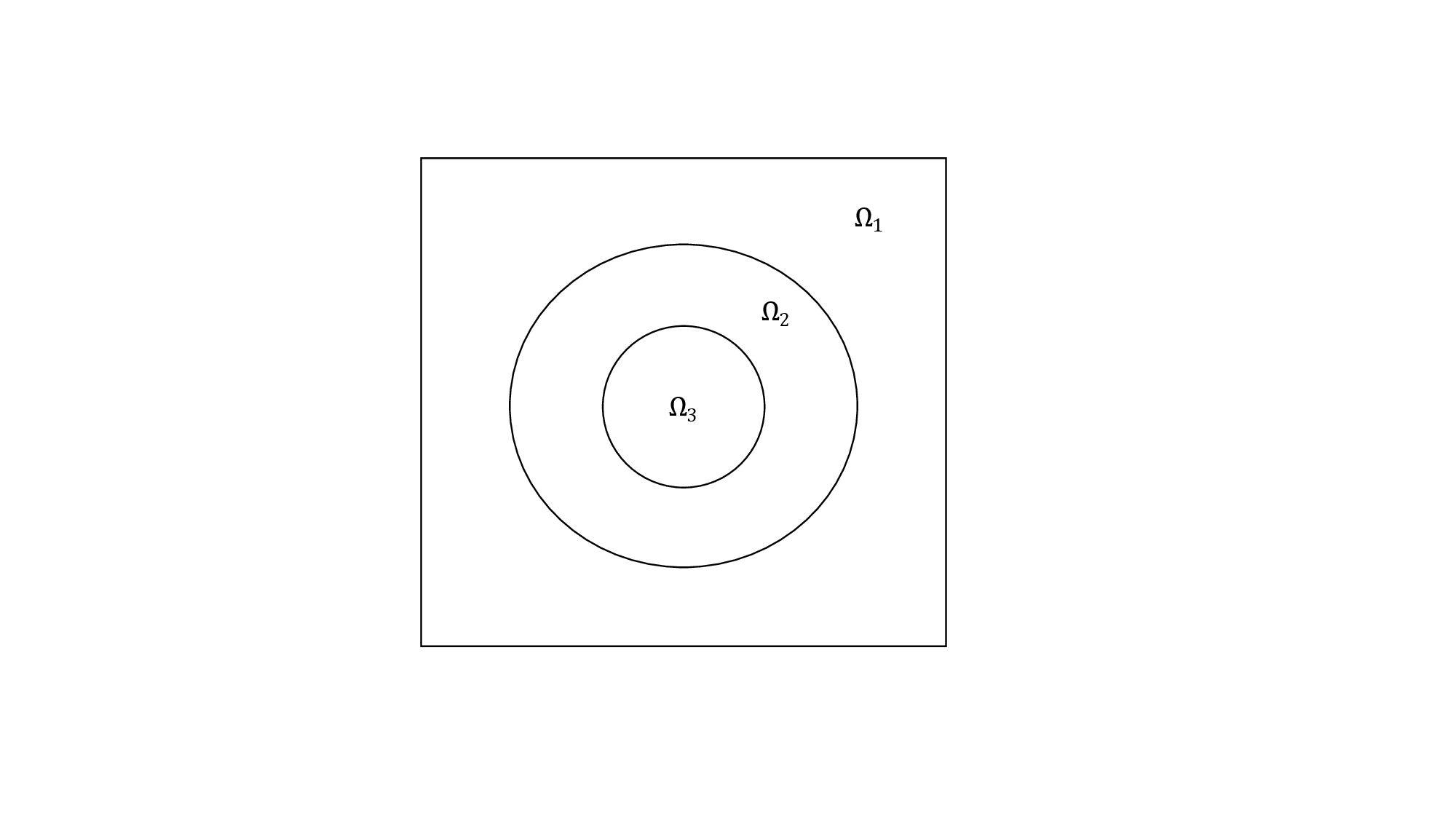}
\caption{Example of regions corresponding to Lemma \ref{decompositionlem2}}\label{fig5}
\end{figure}
Then we have $\hat\Omega^1=\{x:|x_i|<3,\ i=1,\dots n\}$, $\hat\Omega^2=\{x:|x|<2\}$ and $\hat\Omega^3=\Omega_3$. We also set $\tilde u_1(x)=1,\ x\in\Omega_1$, $\tilde u_2(x)=3,\ x\in\Omega_2$ and $\tilde u_3(x)=2,\ x\in\Omega_3$. Then by the construction of functions in the proof of Lemma \ref{decompositionlem} we obtain
\begin{align*}
&v_1(x)=1,\ x\in\hat\Omega^1,\\
&\tilde u_1^1(x)=0,\ \tilde u_2^1(x)=2,\ \tilde u_3^1(x)=1,\\
&v_2(x)=2,\ x\in \hat \Omega^2,\\
&\tilde u_1^2(x)=0,\ \tilde u_2^2(x)=0,\ \tilde u_3^2(x)=-1,\\
&v_3(x)=-1,\ x\in \hat \Omega^3,\\
&\tilde u_1^3(x)=0,\ \tilde u_2^3(x)=0,\ \tilde u_3^3(x)=0.\\
\end{align*}
\end{remark}

Let $X$ be a Hilbert space and $Y$ be a closed subspace of $X$. For $u\in X$ we set
$$\mathrm{dist}\, (u,Y):=\inf_{w\in Y}\lVert u-w\rVert_X,$$
where $\lVert \cdot\rVert_X$ is the norm in $X$. The following theorem is the main goal of this section and the main ingredient of the proof of Theorems \ref{mainthm} and \ref{mainthm2}. 

\begin{theorem}\label{distthm}
Assume (A) (resp., (A')) and set $X=\bigoplus_{\alpha=1}^NH^2(\Omega_{\alpha})$ and $Y=H^1(\mathbb R^n)\cap(\bigoplus_{\alpha=1}^NH^2(\Omega_{\alpha}))$ (resp., $X=\tilde Q_k\cap (\bigoplus_{\alpha=1}^NH^2(\Omega_{\alpha}))$ and $Y=Q_k\cap(\bigoplus_{\alpha=1}^NH^2(\Omega_{\alpha}))$). Then there exists a constant $C>0$ depending only on $\{\Omega_{\alpha}\}$ such that for any $u\in X$ we have
\begin{equation}\label{myeq2.1.0.1}
C^{-1}\mathrm{dist}\, (u,Y)\leq \sum_{\alpha\neq \beta}\lVert \tilde u_{\alpha}|_{\partial \Omega_{\alpha}}-\tilde u_{\beta}|_{\partial\Omega_{\beta}}\rVert_{H^{3/2}(\partial \Omega_{\alpha}\cap\partial\Omega_{\beta})}\leq C\mathrm{dist}\, (u,Y).
\end{equation}
\end{theorem}

\begin{remark}
(1) Theorem \ref{distthm} means that the distance $\mathrm{dist}\, (u,Y)$ is equivalent to the degree of discontinuity on the intersections of the boundaries. Note also that these values are semi-norms and the result can be regarded as equivalence of the semi-norms. We can confirm that $\mathrm{dist}\, (u,Y)$ is a semi-norm as follows. The conditions $\mathrm{dist}\, (u,Y)\geq0$ and $\mathrm{dist}\, (au,Y)=|a|\mathrm{dist}\, (u,Y)$ are obvious. To see $\mathrm{dist}\, (u+v,Y)\leq\mathrm{dist}\, (u,Y)+\mathrm{dist}\, (v,Y)$, for any $\epsilon>0$ choose $w_1,w_2\in Y$ such that $\lVert u-w_1\rVert_X\leq\mathrm{dist}\, (u,Y)+\epsilon,\ \lVert v-w_2\rVert_X\leq\mathrm{dist}\, (v,Y)+\epsilon$, and note $\lVert u+v-(w_1+w_2)\rVert_X\leq\mathrm{dist}\, (u,Y)+\mathrm{dist}\, (v,Y)+2\epsilon$.

(2) Since the true eigenfunctions are in $Y$, it follows from the right inequality that in order to approximate a true eigenfunction by $u\in \bigoplus_{\alpha=1}^NH^2(\Omega_{\alpha})$ we must make the semi-norm $\sum_{\alpha\neq \beta}\lVert \tilde u_{\alpha}|_{\partial \Omega_{\alpha}}-\tilde u_{\beta}|_{\partial\Omega_{\beta}}\rVert_{H^{3/2}(\partial \Omega_{\alpha}\cap\partial\Omega_{\beta})}$ small. From this fact we can see that Theorems \ref{mainthm} and \ref{mainthm2} are plausible. However, it is the left inequality that is used for the proof of the theorems, and it is much more difficult to prove than the right inequality.
\end{remark}

\begin{proof}[Proof of  Theorem \ref{distthm}]
First let us consider the case of (A). Let us prove the left inequality. We decompose $u=(\tilde u_1,\dots,\tilde u_N)\in \bigoplus_{\beta=1}^NH^2(\Omega_{\beta})$ as $u=\sum_{\alpha=1}^Nv_{\alpha},\ v_{\alpha}\in H^2(\hat\Omega^{\alpha})$ as in Lemma \ref{decompositionlem}. We decompose further $v_{\alpha}$ as $v_{\alpha}=\psi_{\alpha}+\phi_{\alpha}$,
\begin{align*}
\psi_{\alpha}&:=(-\Delta^D+1)^{-1}(-\Delta+1)v_{\alpha},\\
\phi_{\alpha}&:=v_{\alpha}-(-\Delta^D+1)^{-1}(-\Delta+1)v_{\alpha},
\end{align*}
where $-\Delta^D$ is the Dirichlet Laplacian in $\hat \Omega^{\alpha}$ and $-\Delta$ means the usual weak derivative.
Then $\psi_{\alpha}\in H^1_0(\hat\Omega^{\alpha})\cap H^2(\hat\Omega^{\alpha})$ and $\phi_{\alpha}\in (H^1_0(\hat\Omega^{\alpha}))^{\perp}\cap H^2(\hat\Omega^{\alpha})$, where $\phi_{\alpha}\in (H^1_0(\hat\Omega^{\alpha}))^{\perp}$ follows from $(-\Delta+1)\phi_{\alpha}=0$. (Note that $\hat\Omega^1=\mathbb R^n$ and $\psi_1=v_1,\ \phi_1=0$ if we label $\Omega_{\alpha}$ as in the proof of Lemma \ref{decompositionlem}.)

We set
$$\Psi:=\sum_{\alpha=1}^N\psi_{\alpha},\qquad \Phi:=\sum_{\alpha=1}^N\phi_{\alpha},$$
where the summation is in the sense of $\bigoplus_{\beta=1}^NH^2(\Omega_{\beta})$.
Then $u=\Psi+\Phi$ and $\Psi\in H^1(\mathbb R^n)\cap(\bigoplus_{\beta=1}^NH^2(\Omega_{\beta}))$. (Here note that since the image of $\psi_{\alpha}$ by the trace operator in $\hat\Omega^{\alpha}$ is $0$, $\psi_{\alpha}$ is extended to an element in $H^1(\mathbb R^n)$ by $\psi_{\alpha}(x)=0$ for $x\in\mathbb R^n\setminus\hat\Omega^{\alpha}$ by Proposition \ref{conectionpro}.)

By Lemma \ref{orthogonallem} it follows that there exists a constant $C>0$  independent of $\Phi$ such that
\begin{equation}\label{myeq2.1.1}
\lVert \Phi\rVert_X\leq \sum_{\alpha=1}^N\lVert \phi_{\alpha}\rVert_{H^2(\hat \Omega^{\alpha})}\leq C\sum_{\alpha=1}^N\lVert \phi_{\alpha}|_{\partial\hat\Omega^{\alpha}}\rVert_{H^{3/2}(\partial\hat\Omega^{\alpha})}.
\end{equation}
Assume $\alpha$ and $\beta$ satisfying $\alpha\neq\beta$ are fixed. Then since $\Psi\in H^1(\mathbb R^n)$, $\phi_{\alpha'}\in H^2(\hat\Omega^{\alpha'})$ and $\phi_{\alpha'}(x)=0,\ x\in\mathbb R^n\setminus\hat\Omega_{\alpha'}$, the difference of restrictions on $\partial\Omega_{\alpha}\cap\partial\Omega_{\beta}$ of $H^2(\Omega_{\alpha})$-component and $H^2(\Omega_{\beta})$-component of $\Psi$ or $\phi_{\alpha'},\alpha'\neq \alpha,\beta$ is $0$. Thus from $u=\Psi+\Phi$ we obtain
\begin{equation}\label{myeq2.1.2}
\sum_{\alpha\neq \beta}\lVert \tilde u_{\alpha}|_{\partial\Omega_{\alpha}}-\tilde u_{\beta}|_{\partial\Omega_{\beta}}\rVert_{H^{3/2}(\partial\Omega_{\alpha}\cap\partial\Omega_{\beta})}=\sum_{\alpha=1}^N\lVert \phi_{\alpha}|_{\partial\hat\Omega^{\alpha}}\rVert_{H^{3/2}(\partial\hat\Omega^{\alpha})}.
\end{equation}
The left inequality in \eqref{myeq2.1.0.1} follows immediately from $\Psi\in H^1(\mathbb R^n)$, \eqref{myeq2.1.1} and \eqref{myeq2.1.2}.

As for the right inequality, since $\tilde w_{\alpha}|_{\partial \Omega_{\alpha}}-\tilde w_{\beta}|_{\partial \Omega_{\beta}}=0$ on $\partial \Omega_{\alpha}\cap\partial\Omega_{\beta}$ for $w=(\tilde w_1,\dots,\tilde w_N)\in Y\subset X$, there exists a constant $C_{\alpha\beta}$ depending only on $\Omega_{\alpha}$ and $\Omega_{\beta}$ such that
\begin{align*}
&\lVert \tilde u_{\alpha}|_{\partial\Omega_{\alpha}}-\tilde u_{\beta}|_{\partial\Omega_{\beta}}\rVert_{H^{3/2}(\partial\Omega_{\alpha}\cap\partial\Omega_{\beta})}\\
&\quad=\lVert (\tilde u_{\alpha}-\tilde w_{\alpha})|_{\partial\Omega_{\alpha}}-(\tilde u_{\beta}-\tilde w_{\beta})|_{\partial\Omega_{\beta}}\rVert_{H^{3/2}(\partial\Omega_{\alpha}\cap\partial\Omega_{\beta})}\\
&\quad\leq C_{\alpha\beta}(\lVert \tilde u_{\alpha}-\tilde w_{\alpha}\rVert_{H^2(\Omega_{\alpha})}+\lVert \tilde u_{\beta}-\tilde w_{\beta}\rVert_{H^2(\Omega_{\beta})}),
\end{align*}
for any $w\in Y$, where we used the boundedness of the trace operator in the last inequality. The right inequality in \eqref{myeq2.1.0.1} follows immediately from this inequality.
The case of (A') is proved in exactly the same way using Lemma \ref{decompositionlem2} instead of Lemma \ref{decompositionlem}.
\end{proof}

Finally we need a lemma about orthonormalization of a basis in a Hilbert space. For a vector $x\in\mathbb C^M$ we define norms by $\lVert x\rVert_1:=\sum_{i=1}^M|x_i|,$ $\lVert x\rVert_2:=(\sum_{i=1}^M|x_i|^2)^{1/2}$ and $\lVert x\rVert_{\infty}:=\max_{1\leq i\leq M}|x_i|$. Correspondingly to $\lVert x \rVert_l$ we define the matrix norm for an $M\times M$ matrix $A$ by $\lVert A\rVert_l:=\sup_{x\neq 0}\frac{\lVert Ax\rVert_l}{\lVert x\rVert_l}$. It is well known that $\lVert A\rVert_1=\max_{1\leq j\leq M}\sum_{i=1}^M|a_{ij}|$ and $\lVert A\rVert_{\infty}=\max_{1\leq i\leq M}\sum_{j=1}^M|a_{ij}|$.

\begin{lemma}\label{Schmidt}
Let $\Psi^i,\dots,\Psi^M$ be linearly independent vectors in a Hilbert space $Z$ such that if we define a matrix $\mathcal E=(\epsilon_{ij})$ by
$$\langle \Psi^i,\Psi^j\rangle =\delta_{ij}+\epsilon_{ij},$$
we have $2\lVert\mathcal E\rVert_1+\epsilon_{\max}<1$, where $\epsilon_{\max}:=\max_{1\leq i\leq M}|\epsilon_{ii}|$. Then there exists an $M\times M$ matrix $B$ such that $(\hat\Psi^1,\dots,\hat\Psi^M)^T:=B(\Psi^1,\dots,\Psi^M)^T$ satisfies $\langle \hat\Psi^i,\hat\Psi^j\rangle=\delta_{ij}$ and $\lVert B-I_M\rVert_{\infty}<(2\lVert\mathcal E\rVert_1+\epsilon_{\max})(1-2\lVert\mathcal E\rVert_1-\epsilon_{\max})^{-1}$.
\end{lemma}

\begin{proof}
We use the Schmidt orthonormalization.
The Schmidt orthonormalization can be written as
\begin{equation}\label{myeq2,2}
\begin{split}
\tilde\Psi^1&:=\Psi^1,\\
\tilde\Psi^i&:=\Psi^i-\sum_{1\leq j,k\leq i-1}\{(A_{i-1}^{-1})_{jk}\langle\Psi^{k},\Psi^i\rangle\}\Psi^{j},\ 2\leq i\leq M,\\
\hat\Psi^i&:=\lVert \tilde\Psi^i\rVert^{-1}\tilde\Psi^i,\ 1\leq i\leq M,\\
\end{split}
\end{equation}
where $A_{i-1}$ is the $(i-1)\times(i-1)$ matrix defined by $(A_{i-1})_{jk}:=\langle\Psi^{j},\Psi^{k}\rangle,\ 1\leq j,k\leq i-1$, and $(A_{i-1}^{-1})_{jk}$ denotes the $(j,k)$-component of $A_{i-1}^{-1}$  and $\lVert \cdot\rVert$ is the norm in the Hilbert space $Z$. (Note that the orthogonal projection of $\Psi$ onto $\mathcal L(\Psi^1,\dots,\Psi^{i-1})$ is given by $\sum_{1\leq j,k\leq i-1}\{(A_{i-1}^{-1})_{jk}\langle\Psi^{k},\Psi\rangle\}\Psi^{j}$.) We set $M\times M$ matrices $\tilde B=(\tilde b_{ij})$, $F$ and $B$ by
$$\tilde b_{ij}:=\sum_{k=1}^{i-1}(A_{i-1}^{-1})_{jk}\langle\Psi^{k},\Psi^i\rangle,$$
$$F:=\mathrm{diag}\, [\lVert \tilde\Psi^1\rVert^{-1},\dots,\lVert \tilde\Psi^M\rVert^{-1}],$$
and $B:=F(I_M-\tilde B)$, where $I_M$ is the $M\times M$ identity matrix. Then $(\hat\Psi^1,\dots,\hat\Psi^M)^T\newline:=B(\Psi^1,\dots,\Psi^M)^T$ is obvious from \eqref{myeq2,2}, and it remains to prove the estimate of the norm of $B$. 

Let us define $M\times M$ matrices $\tilde A'=(\tilde a'_{jk})$ and $A'=(a'_{jk})$ by $\tilde a'_{jk}:=|\epsilon_{jk}|$ and $A':=\sum_{m=0}^{\infty}(\tilde A')^m=(I_M-\tilde A')^{-1}$ respectively. Then it is easily seen that $\lVert A'\rVert_1\leq (1-\lVert \mathcal E\rVert_1)^{-1}$.  We also set $\tilde A_{i-1}:=A_{i-1}-I_{i-1}$. Using the Neumann series $A_{i-1}^{-1}=\sum_{m=0}^{\infty}(-1)^m\tilde A_{i-1}^m$, noting $(\tilde A_{i-1})_{jk}=\epsilon_{jk}$ and considering the expression of the components of products $\tilde A_{i-1}^m$ of matrices by the components of $\tilde A_{i-1}$, we can easily see that $|(A_{i-1}^{-1})_{jk}|\leq a'_{jk}$ for $1\leq j,k\leq i-1$. Thus we have
\begin{equation*}
\sum_{j=1}^M|\tilde b_{ij}|\leq \sum_{j=1}^M\sum_{k=1}^{i-1}a'_{jk}|\epsilon_{ki}|\leq \lVert A'\rVert_1\lVert \mathcal E\rVert_1\leq \lVert \mathcal E\rVert_1(1-\lVert \mathcal E\rVert_1)^{-1},
\end{equation*}
which implies $\lVert\tilde B\rVert_{\infty}\leq\lVert \mathcal E\rVert_1(1-\lVert \mathcal E\rVert_1)^{-1}$.
Moreover, using $|\lVert \Psi^j\rVert-1|\leq|\lVert \Psi^j\rVert+1||\lVert \Psi^j\rVert-1|=|\lVert \Psi^j\rVert^2-1|=|\epsilon_{jj}|\leq \epsilon_{\max}$ and \eqref{myeq2,2} we obtain
\begin{equation*}
\lVert\tilde \Psi^i-\Psi^i\rVert\leq (1+\epsilon_{\max})\sum_{j=1}^{i-1}|\tilde b_{ij}|\leq (1+\epsilon_{\max})\lVert \mathcal E\rVert_1(1-\lVert \mathcal E\rVert_1)^{-1}.
\end{equation*}
Thus it follows that $\lVert\tilde \Psi^i\rVert\geq (1-2\lVert \mathcal E\rVert_1-\epsilon_{\max})(1-\lVert \mathcal E\rVert_{1})^{-1}$ and
\begin{equation*}
|1-\lVert \tilde \Psi^i\rVert^{-1}|=|\lVert \tilde \Psi^i\rVert-1|\lVert \tilde \Psi^i\rVert^{-1}\leq(\lVert \mathcal E\rVert_1+\epsilon_{\max})(1-2\lVert \mathcal E\rVert_1-\epsilon_{\max})^{-1},
\end{equation*}
where for the estimate of $|\lVert \tilde \Psi^i\rVert-1|$ we used that
$$|\lVert \tilde \Psi^i\rVert-1|=\begin{cases}
\lVert \tilde \Psi^i\rVert-1\leq \lVert \tilde\Psi^i-\Psi^i\rVert +\lVert\Psi^i\rVert-1 &(\lVert \tilde\Psi^i\rVert\geq 1)\\
1-\lVert \tilde \Psi^i\rVert\leq 1 -\lVert\Psi^i\rVert +\lVert \tilde\Psi^i-\Psi^i\rVert   &(\lVert \tilde\Psi^i\rVert< 1)\\
\end{cases},$$
that is, $|\lVert \tilde \Psi^i\rVert-1|\leq \lVert \tilde\Psi^i-\Psi^i\rVert +|\lVert\Psi^i\rVert-1|$.
Hence we can see that $\lVert F-I_M\rVert_{\infty}\leq (\lVert \mathcal E\rVert_{1}+\epsilon_{\max})(1-2\lVert \mathcal E\rVert_1-\epsilon_{\max})^{-1}$ and $\lVert F\tilde B\rVert_{\infty}\leq\lVert \mathcal E\rVert_{1}(1-2\lVert \mathcal E\rVert_1-\epsilon_{\max})^{-1}$. The result follows from these inequalities.
\end{proof}

\section{Proof of the main result}\label{thirdsec}
\begin{proof}[Proof of Theorem \ref{mainthm}]
Since $H=(h_{ij})$ is diagonalized by an $M\times M$ unitary matrix $Y=(y_{ij})$ as $\bar YHY^T=\mathrm{diag}\, [\tilde E_1,\dots,\tilde E_M]$, denoting $Y(u^1,\dots,u^M)^T$ again by $(u^1,\dots,u^M)^T$ we may assume that $H$ is a diagonal matrix from the beginning. However, we need to note that if we denote the old functions by $(u^1_{\mathrm{Old}},\dots,u^1_{\mathrm{Old}})^T$ and the new ones by $(u^1_{\mathrm{New}},\dots,u^1_{\mathrm{New}})^T=Y(u^1_{\mathrm{Old}},\dots,u^1_{\mathrm{Old}})^T$, the discontinuity of the new functions is estimated as
\begin{equation}\label{myeq3.0.0.2}
\begin{split}
&\lVert (\tilde u_{\mathrm{New}}^i)_{\alpha}|_{\partial \Omega_{\alpha}}-(\tilde u_{\mathrm{New}}^i)_{\beta}|_{\partial\Omega_{\beta}}\rVert_{H^{3/2}(\partial \Omega_{\alpha}\cap\partial\Omega_{\beta})}\\
&\quad\leq \sum_{j=1}^M\lVert (\tilde u_{\mathrm{Old}}^j)_{\alpha}|_{\partial \Omega_{\alpha}}-(\tilde u_{\mathrm{Old}}^j)_{\beta}|_{\partial\Omega_{\beta}}\rVert_{H^{3/2}(\partial \Omega_{\alpha}\cap\partial\Omega_{\beta})}.
\end{split}
\end{equation}
If there exists $M'< M$ such that $\tilde E_{M'+1},\dots,\tilde E_M> d$, from the assumption $E_i\leq d$ the inequality of Theorem \ref{mainthm} for these eigenvalues are obvious and we have only to consider $(u^1,\dots,u^{M'})$ and $M'\times M'$ matrix with entries $h_{ij},\ 1\leq i,j\leq M'$. Thus we may assume that $\tilde E_1,\dots,\tilde E_M\leq d$ from the beginning. Since
$$\langle \nabla u^i,\nabla u^i\rangle+\langle u^i, Vu^i\rangle=\tilde E_i\leq d,$$
by the assumption (B) we have
$$(1-a)\lVert\nabla u^i\rVert^2-b\lVert u^i\rVert^2\leq d,$$
that is,
$$\lVert \nabla u^i\rVert\leq \left(\frac{d+b}{1-a}\right)^{1/2},$$
where $\lVert w\rVert:=\langle w,w\rangle^{1/2}$.
By Theorem \ref{distthm} we can decompose $u^i$ as $u^i=\Psi^i+\Phi^i$, where $\Psi^i\in H^1(\mathbb R^n)\cap(\bigoplus_{\alpha=1}^NH^2(\Omega_{\alpha}))$ and
\begin{equation}\label{myeq3.0.0.1}
\lVert \Phi^i\rVert_{\bigoplus_{\alpha=1}^NH^2(\Omega_{\alpha})}\leq C\sum_{\alpha\neq \beta}\lVert \tilde u_{\alpha}^i|_{\partial \Omega_{\alpha}}-\tilde u_{\beta}^i|_{\partial\Omega_{\beta}}\rVert_{H^{3/2}(\partial \Omega_{\alpha}\cap\partial\Omega_{\beta})}.
\end{equation}
Henceforth, we denote constants independent of $M$ and $u^1,\dots,u^M$ by $C$.
Then $u^i=\Psi^i+\Phi^i$ yields
$$\lVert\nabla \Psi^i\rVert\leq C+\lVert \nabla\Phi^i\rVert,$$
where $C:=\left(\frac{d+b}{1-a}\right)^{1/2}$. 
Moreover, from $\lVert u^i\rVert=1$ it follows that
$$\lVert \Psi^i\rVert\leq 1+\lVert \Phi^i\rVert.$$

We define an $M\times M$ matrix $\check H=(\check h_{ij})$ by
$$\check h_{ij}:=\langle \nabla\Psi^i,\nabla\Psi^j\rangle +\langle\Psi^i,V\Psi^j\rangle.$$
Then we have 
\begin{equation}\label{myeq3.0.0.3}
\lVert H-\check H\rVert_{\infty}\leq C((1+\epsilon')\epsilon''+\epsilon'M),
\end{equation}
where $\epsilon':=\max_{1\leq i\leq M}(\lVert\Phi^i\rVert+\lVert\nabla\Phi^i\rVert)$ and $\epsilon'':=\sum_{j=1}^M(\lVert\Phi^j\rVert+\lVert\nabla\Phi^j\rVert)$.
By Lemma \ref{Schmidt} if 
\begin{equation}\label{myeq3.0.1}
3((1+\epsilon')\epsilon''+\epsilon'\sqrt M)<1/2,
\end{equation}
we can see that there exists an $M\times M$ matrix $B$ such that $(\hat\Psi^1,\dots,\hat\Psi^M)^T:=B(\Psi^1,\dots,\Psi^M)^T$ satisfies $\langle \hat\Psi^i,\hat\Psi^j\rangle=\delta_{ij}$, and $\lVert B-I_M\rVert_{\infty}\leq C(\epsilon''+\epsilon'\sqrt M)$. Here we used that
$$\langle\Psi^i,\Psi^j\rangle=\langle\Psi^i,(\Psi^j-u^j)\rangle+\langle(\Psi^i-u^i),u^j\rangle+\langle u^i,u^j\rangle=-\langle\Psi^i,\Phi^j\rangle-\langle\Phi^i,u^j\rangle+\delta_{ij},$$
and $\sum_{j=1}^M|\langle\Phi^i,u^j\rangle|\leq\sqrt M\lVert\Phi^i\rVert$ by the Bessel inequality and the Cauchy-Schwarz inequality.
We define $\hat H=(\hat h_{ij})$ by
$$\hat h_{ij}:=\langle \nabla\hat\Psi^i,\nabla\hat\Psi^j\rangle +\langle\hat\Psi^i,V\hat\Psi^j\rangle.$$
Then $\hat H=\bar B\check H B^T$ and $\lVert \hat H-\check H\rVert_{\infty}\leq CM(\epsilon''+\epsilon'\sqrt M)(1+\epsilon''+\epsilon'\sqrt M)$, which yields
\begin{equation}\label{myeq3.0.0.4}
\lVert \hat H-\check H\rVert_{\infty}\leq CM(\epsilon''+\epsilon'\sqrt M),
\end{equation}
under \eqref{myeq3.0.1}.
It follows from \eqref{myeq3.0.0.3} and \eqref{myeq3.0.0.4} that $\lVert H-\hat H\rVert_{\infty}\leq C\tilde\epsilon$ with $\tilde\epsilon:=\epsilon''+\epsilon'M+M(\epsilon''+\epsilon'\sqrt M)$ under \eqref{myeq3.0.1}.
Since $\hat\Psi^i\in H^1(\mathbb R^n)$, by the Rayleigh-Ritz method (cf. Lemma \ref{Rayleigh-Ritz}) the eigenvalues $\hat E_i$ of $\hat H$ are estimated as $\hat E_i\geq E_i,\ i=1,\dots,M$.  By the perturbation theorem of Hermitian matrices (cf. \cite[Problem 1.9.2]{Ch}) and $\lVert \cdot\rVert_2\leq \sqrt M\lVert\cdot\rVert_{\infty}$ which can easily be confirmed, we can also see that $|E_i-\hat E_i|\leq\lVert H-\hat H\rVert_{2}\leq\sqrt M\lVert H-\hat H\rVert_{\infty}$. Thus we obtain
\begin{equation}\label{myeq3.1}
\tilde E_i\geq \hat E_i-C\tilde \epsilon\sqrt M\geq E_i-C\tilde \epsilon\sqrt M.
\end{equation}
The result follows immediately from \eqref{myeq3.1}, \eqref{myeq3.0.0.1} and \eqref{myeq3.0.0.2}.
\end{proof}

The proof of Theorem \ref{mainthm2} is almost the same as that of Theorem \ref{mainthm}. The difference is that we use the case of (A') of Theorem \ref{distthm}. For the estimate $\hat E_i\geq E_i$ we use that $-\Delta+V$ with the domain $\mathcal D_k$ is the selfadjoint operator associated with the quadratic form $\langle \nabla u,\nabla u\rangle_D+\langle u,Vu\rangle_D$ with the form domain $Q_k$, which is proved in the appendix.

\section{Numerical example}\label{fourthsec}
In this section we see how the discontinuity of the function affects evaluation of the eigenvalue in a simple example. We consider the eigenvalue problem $(-\Delta +V)u=Eu$ in $\mathbb R^3$ with
$$V(x):=\begin{cases}
-V_0 & |x|<a\\
0 & |x|\geq a
\end{cases},$$
where $V_0>0$ is a constant. If the function space is restricted to that with angular momentum $0$, the solution is given by
$$u(x)=(4\pi)^{-1/2}r^{-1}\chi(r),$$
where $r:=|x|$, $\chi(r)=A\sin\alpha r,\ r<a$ and $\chi(r)=Ce^{-\beta r},\ r\geq a$ (cf. \cite[Section 15]{Sch}). Here $A, C\in \mathbb R$ are constants and $\alpha:=(V_0-|E|)^{1/2}$, $\beta:=|E|^{1/2}$. We set $\Omega_1:=\{x:|x|>a\}$ and $\Omega_2:=\{x: |x|<a\}$. The condition that the radial function is linearly dependent at $|x|=a$ is given by $\alpha\sin\alpha a-\beta\cos \alpha a=(V_0-|E|)^{1/2}\sin (V_0-|E|)^{1/2}a-|E|^{1/2}\cos(V_0-|E|)^{1/2} a=0$.
The eigenvalue is determined by this condition. If $\frac{\pi^2}{4}<V_0a^2\leq\frac{9\pi^2}{4}$, there exists only one eigenvalue. When $V_0=1$ and $a=\pi$, by solving the equation above numerically, we obtain the eigenvalue $E_1=-0.457591$. The coefficients for the normalized eigenfunction are $A=0.657960$ and $C=4.05791$.

We change $A$ and $C$ (increase $A$ and decrease $C$) under $\lVert u\rVert=1$ and evaluate
$$\tilde E_1=\langle \nabla u,\nabla u\rangle+\langle u,Vu\rangle.$$
Thus $u$ is no longer continuous on the boundary $\partial\Omega_1=\partial\Omega_2=\{x:|x|=a\}$. We obtain the estimate
\begin{equation}\label{myeq4.1}
\tilde E_1\geq-0.457591-0.285781\lVert u|_{\partial\Omega_1}- u|_{\partial\Omega_2}\rVert_{L^2(\partial\Omega_2)}-O(\lVert u|_{\partial\Omega_1}- u|_{\partial\Omega_2}\rVert_{L^2(\partial\Omega_2)}^2).
\end{equation}
The graph of $\tilde E_1$ as a function of $\Gamma:=\lVert u|_{\partial\Omega_1}- u|_{\partial\Omega_2}\rVert_{L^2(\partial\Omega_2)}$ is given in Figure \ref{graph1}.
\begin{figure}[h]
\centering
\includegraphics[width=0.8\textwidth, trim={0 0 0 0}]{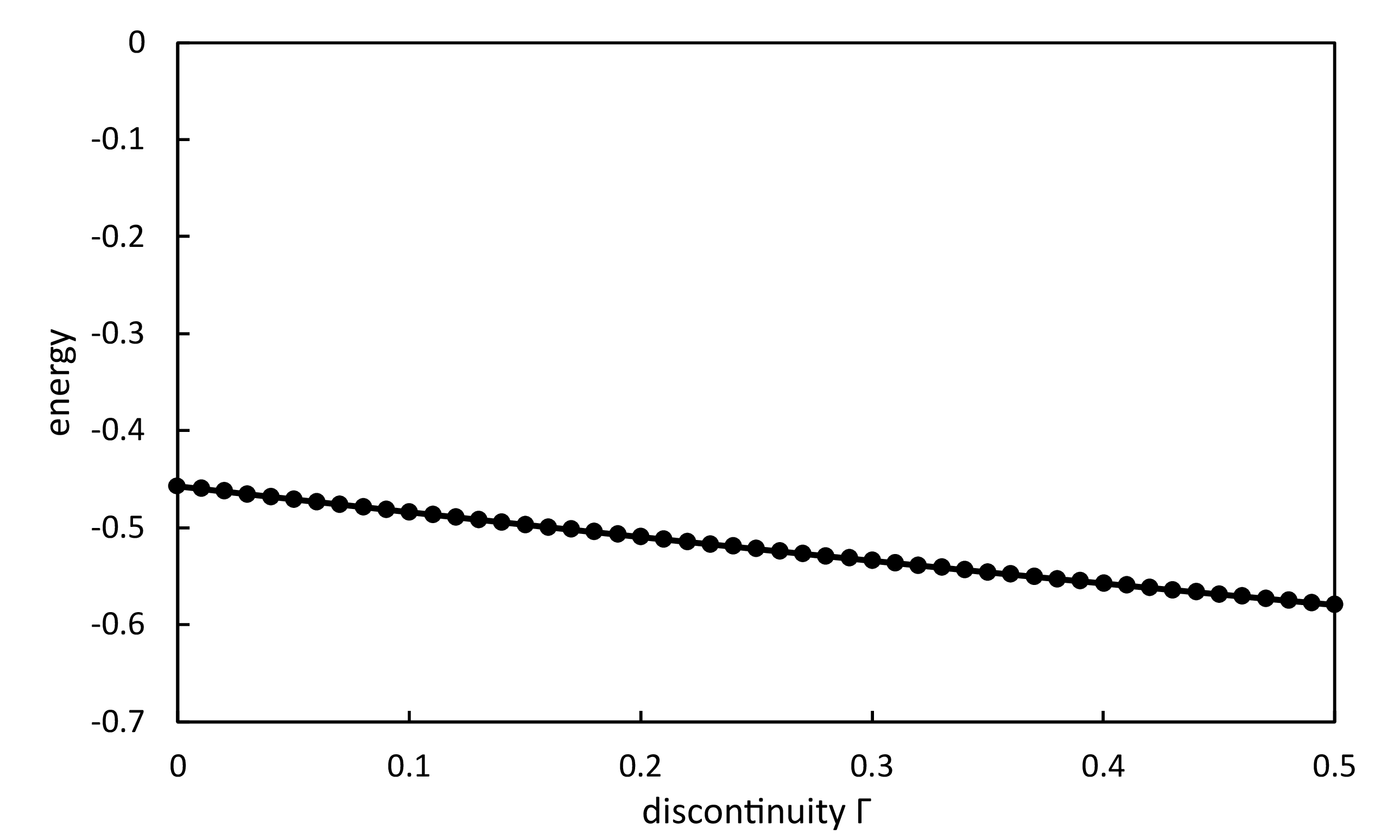}
\caption{Dependence of energy on discontinuity}\label{graph1}
\end{figure}
For example, for $\lVert u|_{\partial\Omega_1}- u|_{\partial\Omega_2}\rVert_{L^2(\partial\Omega_2)}=0.1$ we have
$$\tilde E_1=-0.484263<E_1=-0.457591.$$
Hence the value $\tilde E_1$ can be rather smaller than the true eigenvalue $E_1$. Since by the definition we can easily see that $H^{3/2}$ norm and $L^2$ norm are equivalent for constant valued functions, from \eqref{myeq4.1} we can also see that the order with respect to $\lVert u_{\alpha}|_{\partial\Omega_{\alpha}}- u_{\beta}|_{\partial\Omega_{\beta}}\rVert_{H^{3/2}(\partial\Omega_{\alpha}\cap\partial\Omega_{\beta})}$ in Theorem \ref{mainthm} can not be improved. Note that we investigated only the effect of changes of the coefficients $A$ and $C$, and there are infinitely many ways to change the function $u$ to decrease $\tilde E_1$, if we allow the discontinuity on $\partial\Omega_2$.

\bigskip

\noindent\textbf{Acknowledgment}
This work was supported by JSPS KAKENHI Grant Number JP23K13030.

\def\thesection{\Alph{section}}
\setcounter{section}{0}
\section{Selfadjointness, quadratic form and regularity} 
In this appendix we shall see the reason why a domain of a selfadjoint operator and the form domain of the associated quadratic form are important in evaluation of the eigenvalues of the operator. We also determine the selfadjoint realization of the operator $-\Delta+V$ corresponding to the Bloch function in the unit cell and the form domain of the associated quadratic form. In the last subsection regularity of the Bloch functions is considered.

\subsection{Quadratic form and the Rayleigh-Ritz method}
First we summarize definitions and standard results.
We denote the inner product and the norm in a Hilbert space $X$ by $\langle\cdot,\cdot\rangle$ and $\lVert\cdot\rVert$ respectively. A quadratic form $q$ is called semibounded if there exists $C\in\mathbb R$ such that $q(u,u)\geq C\lVert u\rVert^2$ for any $u\in Q(q)$, where $Q(q)$ is the form domain of $q$ (see e.g. \cite[Section VIII.6]{RS}). An operator $A$ in $X$ is called semibounded if there exists $C\in\mathbb R$ such that $\langle u,Au\rangle\geq C\lVert u\rVert^2$ for any $u\in\mathcal D(A)$, where $\mathcal D(A)$ is the domain of $A$. For the definitions of form core and closed form see e.g. \cite[Section VIII.6]{RS} and \cite[Section VI.1]{Ka}. There exists a one-to-one correspondence between closed semibounded quadratic forms and semibounded selfadjoint operators.

\begin{lemma}[{\cite[Theorem VIII.15 and example 2 in p.277]{RS} and \cite[Theorem 2.1 in Section VI.2]{Ka}}]\label{formop}
Let $q$ be a closed semibounded quadratic form. Then there exists a unique semibounded selfadjoint operator such that
\begin{itemize}
\item[(i)] $\mathcal D(A)\subset Q(q)$ and $q(u,v)=\langle Au,v\rangle$ for every $u\in \mathcal D(A)$ and $v\in Q(q)$;
\item[(ii)] $\mathcal D(A)$ is a form core of $q$;
\item[(iii)] if $u\in Q(q)$, $w\in X$ and $q(u,v)=\langle w,v\rangle$ holds for every $v$ belonging to a form core of $q$, then $u\in\mathcal D(A)$ and $Au=w$.
\end{itemize}
Conversely, if $A$ is a semibounded selfadjoint operator,  there exists a unique closed semibounded quadratic form $q$ such that (i)--(iii) hold.
\end{lemma}
We call the quadratic form $q$ corresponding to $A$ in Lemma \ref{formop} the quadratic form associated with $A$. Conversely, $A$ is called the selfadjoint operator associated with the quadratic form $q$.

The Rayleigh-Ritz method is a useful method to evaluate eigenvalues of selfadjoint operators which can be stated as follows (cf. \cite[Theorem XIII.3]{RS}).
\begin{lemma}\label{Rayleigh-Ritz}
Let $A$ be a semibounded selfadjoint operator that has $M\in\mathbb N$ isolated eigenvalues $E_1,\dots,E_M$ in ascending order at the bottom of its spectrum, and $q$ be the quadratic form associated with $A$. Assume that $u^1,\dots,u^M\in Q(q)$ satisfy $\langle u^i,u^j\rangle=\delta_{ij}$. Define a matrix $H=(h_{ij})$ by $h_{ij}:=q(u^i,u^j)$ and let $\tilde E_1,\dots,\tilde E_M$ be the eigenvalues of $H$ in ascending order. Then
$$E_m\leq \tilde E_m,\ m=1,\dots,M.$$
\end{lemma}

\begin{remark}
The assumption in \cite[Theorem XIII.3]{RS} corresponds to $u^1,\dots,u^M\newline\in \mathcal D(A)$, but it is easily seen that $\mathcal D(A)$ can be replaced by $Q(q)$ as in Lemma \ref{Rayleigh-Ritz} using \cite[Theorem XIII.2]{RS} instead of \cite[Theorem XIII.1]{RS} in the proof.
\end{remark}

It is of crucial importance to choose the quadratic form together with its form domain associated with the selfadjoint operator whose eigenvalues are sought in the Rayleigh-Ritz method, which can be seen from the following example. Let $I:=(0,\pi)$ be an interval. Then the operator $A:=-\partial_x^2$ in $L^2(I)$ with the domain $H^1_0(I)\cap H^2(I)$ or $\{u\in H^2(I): u'(0)=u'(\pi)=0\}$ is a selfadjoint operator. We denote the operator with the first domain by $A^D$ and that with the latter domain by $A^N$. Then the quadratic form associated with $A^D$ (resp., $A^N$) is $q(u,v):=\langle \partial_xu,\partial_xv\rangle_I$ on the form domain $Q^D(q):=H^1_0(I)$ (resp., $Q^N(q):=H^1(I)$). The first eigenvalue of $A^D$ (resp., $A^N$) is $E_1^D=1$ (resp., $E_1^N=0$) and the corresponding eigenfunction is $c\sin x$ (resp., $c'$), where $c$ and $c'$ are normalization constants. When we seek eigenvalues of $A^D$, we can not use $Q^N(q)$ in the Rayleigh-Ritz method, which can be seen from the fact that $q(c',c')=0< E_1^D$ holds for the eigenfunction $c'\in Q^N(q)$ of $A^N$, and therefore, $q(c',c')$ does not give an upper bound of $E^D_1$. Therefore, we need to use the form domain associated with the selfadjoint operator whose eigenvalues are sought. This is the reason why we need to determine the domain of the selfadjoint operator and the form domain of the associated quadratic form corresponding to the Bloch function in the APW method explicitly.

\subsection{Selfadjoint operator and the associated quadratic form for the Bloch function}
In this subsection we prove selfadjointness of $-\Delta+V$ on $\mathcal D_k$ and determine the quadratic form associated with the operator.
We define an operator $-\Delta^k$ in $L^2(D)$ with the domain 
\begin{align*}
\mathcal D(-\Delta^k)=&\mathcal D_k=\Bigg\{u\in H^2(D): \forall j,\ u(x+a_j)=e^{ik\cdot a_j}u(x),\\
&(G_j\cdot\nabla) u(x+a_j)=e^{ik\cdot a_j}(G_j\cdot\nabla) u(x),\ \mathrm{for}\ x=\sum_{l\neq j}c_la_l,\ 0< c_l< 1\Bigg\},
\end{align*}
by $-\Delta^k u:=-\Delta u$, where the right-hand side is the usual weak derivative. Then we have the following.

\begin{proposition}\label{selfak}
\begin{itemize}
\item[(1)] $-\Delta^k$ is a selfadjoint operator with the compact resolvent.\\
\item[(2)] The quadratic form $\langle \nabla u,\nabla u\rangle_{D}$ with the form domain $Q_k$ is the quadratic form associated with the selfadjoint operator $-\Delta^k$.
\end{itemize}
\end{proposition}

There are two ways of proofs. One is that based on the elliptic regularity, and the other is that based on the Fourier analysis (i.e. direct construction of a spectral representation). Here we adopt the Fourier analysis. A proof of the selfadjointness of $-\Delta^k$ based on the Fourier analysis was also referred to in \cite{Wi}.

\begin{proof}
We transform $-\Delta$ by the unitary transform of multiplication $e^{-ik\cdot x}$ as $e^{-ik\cdot x}(-\Delta)e^{ik\cdot x}=-(\nabla+ik)^T(\nabla+ik)$. By this transformation the domain $\mathcal D_k$ is transformed into $\mathcal D_0=\{u\in H^2(D): \forall j,\ u(x+a_j)=u(x),\ (G_j\cdot\nabla) u(x+a_j)=(G_j\cdot\nabla) u(x),\ \mathrm{for}\ x=\sum_{l\neq j}c_la_l,\ 0< c_l< 1\}$. Moreover, the quadratic form $\langle\nabla u,\nabla u\rangle_D$ with the form domain $Q_k$ is transformed into $\langle (\nabla+ik)u,(\nabla+ik)u\rangle_D$ with the form domain $Q_0=\{u\in H^1(D): \forall j,\ u(x+a_j)=u(x),\ \mathrm{for}\ x=\sum_{l\neq j}c_la_l,\ 0< c_l< 1\}$.

Let $B=(b_{ij})$ be the $n\times n$ matrix that maps $D$ onto $(0,2\pi)^n$. Then by the transformation of coordinates by $B$, $-(\nabla+ik)^T(\nabla+ik)$ is transformed into $-(\nabla+ik')^TA(\nabla+ik')$, where $A:=BB^T$ is a positive symmetric matrix and $k':=(B^T)^{-1}k$. The domain $\mathcal D_0$ is transformed into $\hat{\mathcal D}_0:=\{u\in H^2((0,2\pi)^n): \forall j,\ u(x)=u(x+2\pi e_j),\ \partial_{x_j} u(x)=\partial_{x_j} u(x+2\pi e_j),\ \mathrm{for}\ x=(x_1,\dots,x_n),\ x_j=0 \}$, where $e_j$ is the unit vector in the direction of the $j$th axis.  We denote by $\hat{\mathcal H}_0$ the operator $-(\nabla+ik')^TA(\nabla+ik')$ with this domain. Moreover, the quadratic form $\langle \nabla u,\nabla u\rangle_D$ with the form domain $Q_0$ is transformed into $\langle (\nabla+ik')u,A(\nabla+ik')u\rangle_{(0,2\pi)^n}$ with the form domain $\hat Q_0:=\{u\in H^1((0,2\pi)^n): \forall j,\ u(x)=u(x+2\pi e_j),\ \mathrm{for}\ x=(x_1,\dots,x_n),\ x_j=0 \}$.

By the Fourier series expansion, $L^2((0,2\pi)^n)$ is transformed into $l^2(\mathbb Z^n)=\bigotimes_{i=1}^n\newline l^2(\mathbb Z)$, where $l^2(\mathbb Z)$ is the set of all bilateral sequences $\{c_j\}_{j=0,\pm1,\pm2,\dots}$ such that $\sum_{j=-\infty}^{\infty}|c_j|^2<\infty$. Denoting the transformation by $\mathcal F:L^2((0,2\pi)^n)\to l^2(\mathbb Z^n)$, $\hat{\mathcal H}_0$ is transformed into the multiplication operator $\mathcal F \hat{\mathcal H}_0\mathcal F^{-1}=-(im+ik')^TA(im+ik')$ with the domain $\{\{a_{m}\}\in l^2(\mathbb Z^n):\{|m|^2a_{m}\}\in l^2(\mathbb Z^n)\}$, where $m=(m_1,\dots,m_n)\in \mathbb Z^n$ and $a_{m}$ is the Fourier coefficient. Here for the correspondence of the domains, note that $\hat{\mathcal D}_0$ is identified with $H^2(\mathbb T^n)$ (recall Proposition \ref{conectionpro}) and that the $H^2((0,2\pi)^n)$-norm for $u\in H^2(\mathbb T^n)$ corresponds to the norm $\sum_{m\in\mathbb Z^n}|a_{m}|^2(1+|m|^2+|m|^4)$, where $\mathbb T^n$ is the $n$-dimensional torus and $a_{m}$ is the Fourier coefficient of $u$. The operator $\mathcal F \hat{\mathcal H}_0\mathcal F^{-1}$ is obviously selfadjoint and associated with the quadratic form $\sum_{m\in\mathbb Z^n}^{\infty}|a_{m}|^2(-im-ik')^TA(im+ik')$ with the form domain $\{\{a_{m}\}\in l^2(\mathbb Z^n):\{|m|a_{m}\}\in l^2(\mathbb Z^n)\}$. Thus the selfadjointness of $-\Delta^k$ is obvious, and noting that the transform of this form domain by $\mathcal F^{-1}$ is $\hat Q_0$ we can see that (2) holds. The compactness of the resolvent follows from the Rellich's theorem, if we note that the image of a bounded set in $L^2((0,2\pi)^n))$ by $(\hat{\mathcal H}_0+1)^{-1}$ can be regarded as a bounded set in $H^2(\mathbb T^n)$. The proof is complete.
\end{proof}

\begin{proposition}
If $V$ is $-\Delta^k$ bounded with the relative bound smaller than $1$, $-\Delta+V$ in $L^2(D)$ with the domain $\mathcal D_k$ is a selfadjoint operator (denoted by $\mathcal H^k$) with the discrete spectrum. The eigenvalues $E_m(k),\ m=1,2,\dots$ of $\mathcal H^k$ satisfy $E_m(k)\to\infty$ as $m\to\infty$. Moreover, if $E_m(k)$ are labeled in ascending order, each $E_m(k)$ depends analytically on $k$ except at points in a real-analytic subset of $\mathbb R^n$ (i.e. subset expressed as zeros of some real-analytic function) on which $E_m(k)$ is degenerated.
\end{proposition}

\begin{remark}
The analyticity of $E_m(k)$ with respect to $k$ was also proved by \cite[Theorem 1]{Wi} for three-dimensional space $\mathbb R^3$ using the explicit expression of the resolvent kernel of $-\Delta^k$ and the classical Fredholm theory. However, contrary to the comment before Theorem 1 in \cite{Wi}, the analyticity follows easily from the analytic perturbation theory under general settings as in the following proof. As for the exceptional point of the analyticity, existence of a conic singular point for a hexagonal lattice was proved by \cite{Gu}.
\end{remark}

\begin{proof}
By Proposition \ref{selfak}, the relative boundedness, and the Kato-Rellich theorem (see e.g. \cite[Theorem X.12]{RS2}), $-\Delta+V$ is selfadjoint on $\mathcal D_k=\mathcal D(-\Delta^k)$. Moreover, by the relative boundedness and the compactness of the resolvent of $-\Delta^k$, the resolvent of $\mathcal H^k$ is compact, and thus the spectrum of the operator is discrete and the eigenvalues accumulate at $+\infty$ (see e.g. \cite[Theorem XIII.64]{RS4}). As for the analyticity of $E_m(k)$, let $k_0$ be a fixed value. Then by the unitary operator of multiplication by $e^{i(k-k_0)\cdot x}$, $\mathcal H^k$ is transformed into
$$\tilde{\mathcal H}^k:=e^{-i(k-k_0)\cdot x}\mathcal H^ke^{i(k-k_0)\cdot x}=-\Delta-2i(k-k_0)\cdot\nabla+|k-k_0|^2+V,
$$
with the domain $\mathcal D_{k_0}$. Since $2i(k-k_0)\cdot\nabla$ is $-\Delta^k$ bounded, considering the Neumann series of $(\tilde{\mathcal H^k}-z)^{-1}=(\mathcal H^{k_0}-z)^{-1}(I+(\tilde{\mathcal H}^k-\mathcal H^{k_0})(\mathcal H^{k_0}-z)^{-1})^{-1}$, we can see that the resolvent of $\tilde{\mathcal H}^k$ is bounded holomorphic, and thus $\tilde{\mathcal H}^k$ is holomorphic (cf. \cite[Theorem VII-1.3]{Ka}, note that these notions of holomorphy for multivariable $k$ is defined in the same way as in the single variable case). Therefore, by the reduction to the case of a finite-dimensional matrix if we consider the eigenvalues $E_m(k)$ in a bounded open interval $I\subset\mathbb R$ for $k$ in a sufficiently small neighborhood $\mathcal N_{k_0}$ of $k_0$, they are solutions of a characteristic equation $f(k,E)=0$ for $E$, where $f(k,E)$ is a polynomial with respect to $E$ whose coefficients are analytic functions of $k$ (cf. \cite[p.63 and p.370]{Ka}). Hence by the analytic version of the implicit function theorem, we can see that $E_m(k)$ is holomorphic in a sufficiently small neighborhood $\mathcal N_{k_0}$ unless $E_m(k)$ is degenerated at $k_0$ (cf. \cite[p.65]{Ka}). The necessary and sufficient condition that $E_m(k_0)\in I$ is degenerated is that the discriminant $\omega(k)$ of $f(k,E)$ vanishes at $k=k_0$ (cf. \cite[p.101]{Wa}). Since the discriminant $\omega(k)$ is a real-analytic function of $k$, the set $\{k\in\mathcal N_{k_0}:\omega(k)=0\}$ is a real-analytic subset, which completes the proof.
\end{proof}

\begin{proposition}
If $V$ is $-\Delta^k$ bounded with relative bound smaller than $1$ and there exist constants $0\leq a<1,\ b\geq 0$ such that
$$|\langle u,Vu\rangle_D|\leq a\langle \nabla u,\nabla u\rangle_D+b\langle u,u\rangle_D,$$
for any $u\in Q_k$, the selfadjoint operator $\mathcal H^k$ is associated with the quadratic form $\langle\nabla u,\nabla u\rangle_D+\langle u,Vu\rangle_D$ with the form domain $Q_k$.
\end{proposition}

\begin{proof}
By the KLMN theorem (see e.g. \cite[Theorem X.17]{RS2}) there exists a selfadjoint operator $\mathcal G^k$ associated with the quadratic form
\begin{equation}\label{myeqa.2}
\langle \nabla u,\nabla v\rangle_D+\langle u,Vv\rangle_D,
\end{equation}
for $u,v\in Q_k$, and $\mathcal D_k$ is a form core for $\mathcal G^k$. By Lemma \ref{formop} (iii) and \eqref{myeqa.2}, we can see that $\mathcal D_k\subset\mathcal D(\mathcal G^k)$ and that $\mathcal G^ku=\mathcal H^ku$ for any $u\in \mathcal D_k$, where $\mathcal D(\mathcal G^k)$ is the domain of $\mathcal G^k$. Since the quadratic form associated with $\mathcal H^k$ also has the form core $\mathcal D_k$, we can see that the quadratic forms associated with $\mathcal G^k$ and $\mathcal H^k$ coincide. This means that $\mathcal G^k=\mathcal H^k$, which completes the proof.
\end{proof}

\subsection{Regularity of the Bloch function}
In the literature, in general the Bloch function is ambiguously referred to as a solution to $-\Delta u+Vu=Eu$ such that $u(x+T)=e^{ik\cdot T}u(x)$ for any $T=l_1a_1+\dotsm + l_na_n,\ l_1,\dots,l_n\in\mathbb Z$. Since the latter condition can be expressed without using the restriction of $u$ to the boundary of the unit cell, we do not need to assume the regularity condition $u\in H^2_{\mathrm{loc}}(\mathbb R^n)$ to define the Bloch function in advance. In this subsection, we prove that if $u\in L^2_{\mathrm{loc}}(\mathbb R^n)$ satisfies $-\Delta u+Vu=Eu$ in the sense of distribution, we have $u\in H^2_{\mathrm{loc}}(\mathbb R^n)$ under an assumption on the potential $V$ which is satisfied by e.g. potentials with Coulomb type singularities. Thus under the assumption, all Bloch functions in the natural sense are obtained as eigenfunctions of selfadjoint operators in the unit cell. This regularity is slightly different from the interior regularity of weak solutions, because we do not assume that $u\in H^1_{\mathrm{loc}}(\mathbb R^n)$ in advance. It should be mentioned here that in \cite{HNS} improved regularity of the eigenfunctions in the unit cell is studied for potentials smooth except at a discrete set on which the potential has Coulomb type singularities.

\begin{proposition}
If there exists a constant $C>0$ such that $\lVert Vu\rVert\leq C(\lVert \nabla u\rVert+\lVert u\rVert)$ for any $u\in H^1(\mathbb R^n)$, then any solution $u\in L^2_{\mathrm{loc}}(\mathbb R^n)$ to $-\Delta u+Vu=Eu$ with some $E\in\mathbb R$ in the sense of distribution satisfies $u\in H^2_{\mathrm{loc}}(\mathbb R^n)$.
\end{proposition}

\begin{proof}
Let $\Omega\subset\mathbb R^n$ be a bounded region. Let us prove $u\in H^1(\Omega)$. We use the mollifier. Let $\rho\in C_0^{\infty}(\mathbb R^n)$ be a function such that $\rho(x)\geq 0$, $\rho(x)=0$ for $|x|\geq 1$, $\int_{\mathbb R^n}\rho(x)=1$ and $\rho(-x)=\rho(x)$. For $\epsilon>0$ we define $\rho_{\epsilon}$ by $\rho_{\epsilon}:=\epsilon^{-n}\rho(\frac{x}{\epsilon})$. For any $\eta\in C_0^{\infty}(\mathbb R^n)$ such that $\eta(x)=1$ for $x\in \Omega$. we have
\begin{align*}
\lVert\eta\nabla(u*\rho_{\epsilon})\rVert^2&=-\langle u*\rho_{\epsilon},\eta^2\Delta(u*\rho_{\epsilon})\rangle-\langle u*\rho_{\epsilon},2\eta(\nabla\eta)\cdot\nabla(u*\rho_{\epsilon})\rangle\\
&\leq \langle(\eta^2(u*\rho_{\epsilon}))*\rho_{\epsilon},(E-V)u\rangle+C_1\lVert (\tilde\eta u)*\rho_{\epsilon}\rVert^2\\
&\quad+c\lVert\eta\nabla(u*\rho_{\epsilon})\rVert^2\\
&\leq (C+|E|)(\lVert \nabla((\eta^2(u*\rho_{\epsilon}))*\rho_{\epsilon})\rVert+\lVert ((\eta^2(u*\rho_{\epsilon}))*\rho_{\epsilon})\rVert)\lVert \tilde \eta u\rVert\\
&\quad+C_2\lVert\tilde \eta u\rVert+c\lVert\eta\nabla(u*\rho_{\epsilon})\rVert^2\\
&\leq C_3\lVert\tilde \eta u\rVert + 2c\lVert\eta\nabla(u*\rho_{\epsilon})\rVert^2,
\end{align*}
where $\tilde \eta\in C_0^{\infty}(\mathbb R^n)$ satisfies $\tilde \eta(x)=1$ for $x\in \{x:x=y+z,\ y\in\mathrm{supp}\, \eta,\ |z|\leq \epsilon\}$, and $C_i>0\ ,i=1,2,3$ and $0<c<1/2$ are constants. Thus we obtain $\lVert\eta\nabla(u*\rho_{\epsilon})\rVert^2\leq (1-2c)^{-1}C_3\lVert\tilde \eta u\rVert$. Therefore, there exists a sequence $\epsilon_k$ such that $\partial_{x_{i}}(u*\rho_{\epsilon_k})$ converges weakly to some $v_i\in L^2(\Omega)$ in $\Omega$. Thus $\langle\partial_{x_{i}}(u*\rho_{\epsilon_k}),\varphi\rangle_{\Omega}\to\langle v_i,\varphi\rangle_{\Omega}$ for $\varphi\in C_0^{\infty}(\Omega)$. On the other hand, the left-hand side converges obviously to $-\langle u,\partial_{x_i}\varphi\rangle_{\Omega}$. Hence $u$ is weakly differentiable and $\partial_{x_i}u=v_i$ in $\Omega$, which implies $u\in H^1_{\mathrm{loc}}(\mathbb R^n)$. In a similar way, we can also prove that $u\in H^2_{\mathrm{loc}}(\mathbb R^n)$, which completes the proof.
\end{proof}

\bibliographystyle{plain}
\bibliography{sn-bibliography}

\end{document}